\documentclass{cocv}
\usepackage{graphicx}
\usepackage[usenames,dvipsnames]{color}
\usepackage{amssymb,amsmath,mathtools}
\usepackage{stmaryrd}
\usepackage{enumerate}
\usepackage{commath}
\usepackage{subfig}
\usepackage{psfrag}

\newtheorem{theorem}{Theorem}[section]
\newtheorem{lemma}[theorem]{Lemma}
\newtheorem{corollary}[theorem]{Corollary}
\newtheorem{proposition}[theorem]{Proposition}

\newtheorem{remark}[theorem]{Remark}
\newtheorem{definition}[theorem]{Definition}
\newtheorem{assumption}{Assumption}

\newcommand{\real}{\mathbb{R}}

\newcommand{\grad}{\nabla}

\newcommand{\bdy}{\partial}

\newcommand{\cone}[1]{{\mathcal C\del{#1}}}

\DeclareMathOperator{\divgText}{div}
\newcommand{\divg}[1]{\divgText{#1}}

\newcommand{\Ltwo}[1]{{L^2\del{#1}}}

\newcommand{\normLt}[2]{\norm{#1}_\Ltwo{#2}}

\newcommand{\normS}[4]{\norm{#1}_\sob{#2}{#3}{#4}}
\newcommand{\normSZ}[4]{\abs{#1}_\sob{#2}{#3}{#4}}

\newcommand{\pair}[1]{\left\langle #1 \right\rangle}

\newcommand{\pairLt}[2]{{\mathop{\pair{#1}}}_{\Ltwo{#2}\times\Ltwo{#2}}}

\renewcommand{\restriction}{\vert}

\newcommand{\sob}[3]{{W^{#1}_{#2}\del{#3}}}
\newcommand{\sobZ}[3]{{\mathring{W}^{#1}_{#2}\del{#3}}}

\newcommand {\cB} {{\mathcal B}}

\newcommand {\cI} {{\mathcal I}}

\newcommand {\cJ} {{\mathcal J}}
\newcommand {\cL} {{\mathcal L}}

\newcommand {\cN} {{\mathcal N}}

\newcommand {\cQ} {{\mathcal Q}}
\newcommand {\cR} {{\mathcal R}}
\newcommand {\cT} {{\mathcal T}}

\newcommand {\cS} {{\mathcal S}}

\newcommand {\adjoint} {\varphi}
\newcommand {\yop} {{\overline y}}
\newcommand {\uop} {{\overline u}}
\newcommand {\adjop} {{\overline \adjoint}}

\begin{document}
\title{Optimal Control Of Surface Shape}\thanks{The research of Shawn W. Walker was supported by NSF grant DMS-1418994.}

\author{Harbir Antil}\address{Department of Mathematical Sciences. George Mason University, Fairfax, VA 22030, USA. Email: {\tt hantil@gmu.edu}. Phone: 703-993-5777.
Fax: 703-993-1491}.
\author{Shawn W.~Walker}\address{Department of Mathematics and Center for Computation and Technology (CCT), Louisiana State University, Baton Rouge, LA 70803, USA. Email: {\tt walker@math.lsu.edu}. Phone: 225-578-1603. Fax: 225-578-4276}.

\date{December XX, 2014}

\begin{abstract}
Controlling the shapes of surfaces provides a novel way to direct self-assembly of colloidal particles \emph{on those surfaces} and may be useful for material design.  This motivates the investigation of an optimal control problem for surface shape in this paper.  Specifically, we consider an objective (tracking) functional for surface shape with the prescribed mean curvature equation in graph form as a state constraint.  The control variable is the prescribed curvature.  We prove existence of an optimal control, and using improved regularity estimates, we show sufficient differentiability to make sense of the first order optimality conditions.  This allows us to rigorously compute the gradient of the objective functional for both the continuous and discrete (finite element) formulations of the problem.  Moreover, we provide error estimates for the state variable and adjoint state.  Numerical results are shown to illustrate the minimizers and optimal controls on different domains.
\end{abstract}
\begin{resume} ... \end{resume}
\maketitle

\section*{Introduction}
\label{sec:intro}

Directed and self-assembly of micro and nano-structures is a growing research area with applications in material design \cite{Furst_PNAS2011, Vitelli_SM2013, Wang_APL2008}.  Controlling surface geometry can be beneficial for directing the assembly of micro-structures (colloidal particles) \cite{Irvine_Nature2010}.  This is because there is a coupling between the geometry of surfaces/interfaces and the arrangements of charged colloidal particles, or polymers, on those curved surfaces \cite{Lipowsky_PhysA1998, Zakhary_SM2013}; in particular, the presence of defects can seriously affect the surface geometry \cite{Irvine_SM2012, Irvine_Nature2010} and vice-versa.  Moreover, experimental techniques have been developed for creating ``custom shapes'' (from swell gels) by encoding a desired surface metric \cite{RMWilson_2012a}.

With the above motivation, we investigate an optimal PDE control problem which controls the surface shape by prescribing the mean curvature. We consider an open, bounded, $C^{1,1}$ domain $\Omega \subset \real^{n}$ for an embedded surface in $\real^{n+1}$, with boundary of $\Omega$ denoted by $\bdy\Omega$ and $n \ge 1$. If $X_1$ and $X_2$ are two Banach spaces, then $X_1 \hookrightarrow X_2$ and $X_1 \subset \subset X_2$ denote the continuous and compact embeddings of $X_1$ in $X_2$ respectively. $\sob{1}{p}\Omega$, $1 \le p \le \infty$ defines the standard Sobolev space with corresponding norm $\normS{\cdot}1p\Omega$. Moreover, $\sobZ{1}{p}\Omega$ indicates the Sobolev space with zero trace and $\sob{-1}{p'}\Omega$ is the canonical dual of $\sobZ{1}{p}\Omega$, for $1 \le p < \infty$, such that $1/p+1/{p'}=1$.  In deriving various inequalities and estimates, we pay special attention to the constants, $C$, involved.

Then we are interested in solving the following PDE-constrained optimization problem:
\begin{equation} \label{eq:cost}
    \inf \cJ\del{y,u} :=
    \frac{1}{2}\norm{y-y_d}_{L^2(\Omega)}^2
      + \frac{\alpha}{2} \norm{u}_{L^2(\Omega)}^2  \quad \mbox{over }
        y - v \in \sobZ1\infty{\Omega},u \in U_{ad} .
\end{equation}
subject to
\begin{equation} \label{eq:state}
    - \divg{\frac{\grad y}{\cQ(y)}} - u = 0  \quad \mbox{ in } \Omega .
\end{equation}
%
%
%
The second order nonlinear operator in \eqref{eq:state} describes the mean curvature in graph form, where $y$ is the height function, and $\cQ(y) = \del{1+\abs{\grad y}^2}^{1/2}$ denotes the surface measure.  Moreover, we have an integral constraint on $u$: for some fixed $p > n$ and fixed $\theta > 0$, $u$ is in the convex set
\begin{equation*}
    U_{ad} := \set{u \in L^2(\Omega) : \int_{\Omega} |u|^p \leq \theta^p}, \quad \text{(see Definition~\ref{def:admissible_control})}.
\end{equation*}
Eventually, see Remark \ref{rem:U_set_nonempty} and \thmref{thm:state_strong}, we will show there exists a value of $\theta$ for which $U_{ad}$ is not empty.  Note: throughout the entire paper, we now fix $p$ to a value strictly greater than $n$.
In principle, either $u$ or $v$ (boundary value) may act as a control variable, but in this work we will assume that $v$ is a fixed given function and $u$ is the control variable.

We emphasize that the mean curvature operator in \eqref{eq:state} is only locally coercive \cite[P. 104]{DKindelehrer_GStampacchia_1980}, which makes this problem harder than it appears. For instance, a compatibility condition between the domain $\Omega$ and right-hand-side $u$ must hold for \eqref{eq:state} to have a solution \cite{MGiaquinta_1974}.  For instance, integrating both sides of \eqref{eq:state} leads to
\begin{equation*}
  \left| \int_\Omega u \right| = \left| \int_\Omega \divg{\frac{\grad y}{\cQ(y)}} \right| =
  \left| \int_{\partial \Omega} \nu \cdot \frac{\grad y}{\cQ(y)} \right| \leq \int_{\partial \Omega} \left| \nu \cdot \frac{\grad y}{\cQ(y)} \right| \leq \int_{\partial \Omega} 1,
\end{equation*}
where $\nu$ is the outer unit normal of $\partial \Omega$.  Clearly $u$ cannot be too large if \eqref{eq:state} is to be meaningful; in fact, the compatibility condition is even more involved \cite{MGiaquinta_1974}.  Thus, \eqref{eq:state} is intricate, even for ``nice'' domains.

The control of mean curvature \eqref{eq:state} and similar operators in full generality has not been dealt with before. The closest approach is in \cite{HAntil_RHNochetto_PSodre_2014a, HAntil_RHNochetto_PSodre_2013a} where they study the control of a Laplace free boundary problem with surface tension effect for $n=1$. This amounts to solving a Laplace equation in the bulk which is a subset of $\real^{2}$ and the prescribed mean curvature equation \eqref{eq:state} on $(0,1) \subset \real^1$ for an embedded surface in $\real^{2}$. Furthermore, they replaced the curvature operator by a simpler version, i.e.
\begin{equation}\label{eq:curvature_operator_lin}
    \frac{\Delta y}{\cQ(y)}.
\end{equation}
In the present paper, we work in domains $\Omega \subset \real^n$, with $n \ge 2$, and we do not use the simplied curvature operator \eqref{eq:curvature_operator_lin}, i.e. we consider the general nonlinear operator \eqref{eq:state}. The second novelty of this paper is the proof of the existence of a strong unique solution to \eqref{eq:state}: for a given $u \in L^p(\Omega)$, $p > n$, if $u \in U_{ad}$, and $v \in W^2_p(\Omega)$, we prove that $y \in W^2_p(\Omega)$ (see, \thmref{thm:state_strong}). We remark that no smallness condition is assumed on the boundary data $v$. We use an implicit function theorem (IFT) \cite[2.7.2]{LNirenberg_1974a} based framework to prove this result.
This is an improvement over previous results in \cite{PAmster_MCMariani_2001a,PAmster_MCMariani_2003a}. The improvement being that in \cite[Theorem~1]{PAmster_MCMariani_2003a}, Amster et al
require $v \in \sob2p\Omega$ to be small enough. Moreover they use the Schauder theorem to show the existence and therefore $y$ may lack uniqueness.
The implicit function theorem framework not only gives us the existence and uniqueness but also the Fr\'echet differentiability of our control to state map \cite[Section~1.4.2]{MHinze_RPinnau_MUlbrich_SUlbrich_2009a}; the latter is crucial to derive the first order necessary optimality system.
In addition, by further assuming a smallness condition on the data $v$, we derive a continuity estimate for the solution to the state equation \eqref{eq:state} in \thmref{thm:fixed_point}.

The importance of such a continuity estimate is well-known in the literature; see \cite[Page~97]{DKindelehrer_GStampacchia_1980} for the obstacle problem with locally coercive-operators where a similar result leads to well-posedness. We will exploit this result to prove that the control-to-state map, in \lemref{lem:S_Lipschitz}, is Lipschitz continuous. This Lipschitz continuity will be used to prove the Lipschitz continuity of the Fr\'echet derivative of the control-to-state map in \lemref{lem:Sp_Lipschitz}. This is crucial to deal with the \emph{norm-discrepancy} in \lemref{lem:aux_estimate}, which then allows us to prove the quadratic growth condition in Corollary~\ref{cor:quad_growth}. Later, we utilize Corollary~\ref{cor:quad_growth} to show the order of convergence for the optimal control when discretized using finite element methods.

To summarize, we do not need a smallness assumption on $v$ but only on $u$ to prove the existence and uniqueness of a $W^2_p$ solution to \eqref{eq:state} within an IFT framework (\thmref{thm:state_strong}). However, for the remaining paper, we need a smallness condition on both $v$ and $u$. As pointed out earlier, such a condition on $u$ appears naturally due to the structure of \eqref{eq:state}. However, at first glance, the condition on $v$ might seem unnecessary. We would like to stress that without this additional assumption on the data $v$, using the  techniques developed in this paper, it is not possible to show the crucial $W^2_p$ a priori estimate for the solution to \eqref{eq:state}.

We discretize all the quantities using piecewise linear finite elements. For $n=2$, and using the continuity estimate, we derive an a priori finite element error estimate for the state following \cite{RHNochetto_1989a}. Invoking the discrete inf-sup conditions, we derive an a priori error estimate for the associated adjoint solution.
We extend a projection argument from \cite[Theorem~6.1]{HAntil_RHNochetto_PSodre_2013a} which, in conjunction with second order sufficient conditions, gives us a quasi-optimal a priori error estimate for the control.  If the control is discretized by piecewise constant finite elements, then the error estimate is optimal.


\section{The State Equation}

\subsection{Weak solution}

For a Lipschitz domain $\Omega$ and $v$ in $L^1(\Omega)$, Giaquinta in \cite{MGiaquinta_1974} gives a necessary and sufficient condition for the existence of a solution $y$ in the space of functions of bounded variation (BV) for the state equation \eqref{eq:state}.  In \thmref{thm:state_W11}, we state another existence result which says that if $v$ is slightly more regular, then $y$ is more regular as well.
\begin{theorem}[$W^1_1$ state]
\label{thm:state_W11}
Let $\Omega$ be Lipschitz and $v \in \sob11\Omega$. Then there exists an open set $U_1 \subset \sob{-1}\infty\Omega$, with $0 \in U_1$, such that for every $u \in U_1$, there exists a unique solution $y - v \in \sobZ{1}{1}{\Omega}$ solving \eqref{eq:state}.
\end{theorem}
\begin{proof}
See \cite[P. 351]{IEkeland_RTemam_1999a}.
\end{proof}

\thmref{thm:state_W11} further implies that for a given $u \in U_1$, there exists a unique
$y - v \in \sobZ11\Omega$ satisfying the state equation \eqref{eq:state}
in variational form
\begin{equation}
\label{eq:state_weak}
    \cB(y,w) = \varphi(w) \quad \mbox{ for all }
              w \in \sobZ11\Omega ,
\end{equation}
where
$\cB(y,w) := \int_\Omega \frac{\grad y}{\cQ(y)} \grad w$ and
$\varphi(w) := \pair{u,w}_{\sob{-1}{\infty}\Omega,\sobZ11\Omega}$, and $\pair{\cdot,\cdot}_{\sob{-1}\infty\Omega,\sobZ11\Omega}$ indicates the duality pairing.

\begin{remark}
We remark that for the existence of solutions in $W^1_1$, the standard PDE theory for linear equations only requires the data $u$ to be in $\sobZ1\infty\Omega^*$ \cite[Theorem~2.2]{AErn_JLGuermond_2006a}.  But \thmref{thm:state_W11} implies that given $u \in U_1$, which further belongs to $\sob{-1}\infty\Omega \subset \sobZ1\infty\Omega^*$, it is actually more regular. It might be possible to exploit this fact to prove that for $v \in \sob1\infty\Omega$, the solution $y - v \in \sobZ1\infty\Omega$.  For this to be true, our approach in \thmref{thm:state_strong} would require $\Delta$ (Laplacian operator) to be an isomorphism from $\sobZ1\infty\Omega$ to $\sob{-1}\infty\Omega$, which is not clear.
\end{remark}

For subsequent sections, we rewrite \eqref{eq:state_weak} using a nonlinear operator: find $y - v \in \sobZ1\infty\Omega$ satisfying
\begin{equation}
\label{eq:N}
\pair{\cN(y,u), w}_{\sob{-1}\infty\Omega,\sobZ11\Omega} := \cB(y,w) - \varphi(w) = 0 \quad \mbox{for all } w \in \sobZ11\Omega.
\end{equation}
\subsection{Differentiability of $\cN$}

Next we will study some differentiability properties of $\cN$, for the case when $v \in \sob1\infty\Omega$.
\begin{lemma}
\label{lem:Frechet_N_W1inf}
If $v \in \sob1\infty\Omega$, then for every $u \in U_1$, the operator $\cN(\cdot, u) : v\oplus \sobZ1\infty\Omega \rightarrow \sob{-1}\infty\Omega$ is twice Fr\'echet differentiable with respect to $y$ and the first order Fr\'echet derivative at $y \in v\oplus \sobZ1\infty\Omega$ satisfies
\[
    \pair{D_y \cN(y,u)\pair{h}, w}_{\sob{-1}\infty\Omega,\sobZ11\Omega}
     = \pair{ \del{ \cI - \frac{\grad y \grad y^T }{\cQ(y)^2} } \frac{\grad h}{\cQ(y)} ,
         \grad w}_{L^\infty(\Omega),L^1(\Omega)} .
\]
Moreover both the first and second order derivatives are Lipschitz continuous.
\end{lemma}
\begin{proof}
The derivation of $D_y\cN$ is straightforward, so is omitted. We begin by first showing that $\cQ : v\oplus\sobZ1\infty\Omega \rightarrow L^\infty(\Omega)$ is Fr\'echet differentiable. Let $y \in v\oplus\sobZ1\infty\Omega$ and $h \in \sobZ1\infty\Omega$ (note: $y+h \in v\oplus\sobZ1\infty\Omega$). To this end we need to show that for every $\epsilon > 0$ there exists a $\delta > 0$, such that for $\normS{h}1\infty\Omega < \delta$
\[
    \frac{\norm{\cQ(y+h)-\cQ(y)-D_y\cQ(y)\pair{h}}_{L^\infty(\Omega)}}{\normS{h}1\infty\Omega} < \epsilon , \quad \text{where} \quad D_y\cQ(y)\pair{h} = \frac{\grad y}{\cQ(y)} \cdot \grad h.
\]
Define the residual $\cR_1 = \cQ(y+h)-\cQ(y)-D_y\cQ(y)\pair{h}$. Using an algebraic manipulation, we get
\begin{equation}\label{eq:Q_diff}
    \cQ(y+h)-\cQ(y) = \frac{\grad (2y+h) \cdot \grad h}{\cQ(y+h) + \cQ(y)} ,
\end{equation}
whence
\begin{align*}
    \cR_1 = \del{ \frac{\grad (2y+h) }{\cQ(y+h) + \cQ(y)}
           - \frac{\grad y }{\cQ(y)}}\cdot \grad h
      = \frac{\del{\cQ(y)-\cQ(y+h)} \grad y + \cQ(y) \grad h}{\cQ(y)\del{\cQ(y+h) + \cQ(y)}}
           \cdot \grad h    .
\end{align*}
Invoking the $L^\infty$ norm and using the necessary regularity of the underlying terms, we deduce
\[
    \norm{\cR_1}_{L^\infty(\Omega)} \le \del{ \norm{\cQ(y)-\cQ(y+h)}_{L^\infty(\Omega)}  + \normS{h}1\infty\Omega } \normS{h}1\infty\Omega .
\]
It only remains to show that $\cQ$ is a Lipschitz continuous function. In view of \eqref{eq:Q_diff}, for $y,z \in v\oplus \sobZ1\infty\Omega$, $y \neq z$ we get
\begin{equation}\label{eq:Q_Lip}
    \norm{\cQ(y)-\cQ(z)}_{L^\infty(\Omega)}
        \le \norm{\frac{\grad (y+z)}{\cQ(y) + \cQ(z)}}_{L^\infty(\Omega)}
                    \normS{y-z}1\infty\Omega  \le \normS{y-z}1\infty\Omega,
\end{equation}
i.e., $\cQ(\cdot)$ is Fr\'echet differentiable.  Thus, we find that $\norm{\cR_1}_{L^\infty(\Omega)} \le 2 \normS{h}1\infty\Omega^2$.

Next, we use the definition of $\cN$ from \eqref{eq:N} to define the residual $\cR_2 = \cN(y+h,u) - \cN(y,u) - D_y\cN(y,u)\pair{h}$ and write it as
\begin{equation}\label{eqn:Frechet_N_W1inf_pf_1}
    \pair{\cR_2, w}_{\sob{-1}\infty\Omega,\sobZ11\Omega} = \pair{ \overbrace{\frac{\grad (y+h)}{\cQ(y+h)} - \frac{\grad y}{\cQ(y)} - \del{ \cI - \frac{\grad y \grad y^T }{\cQ(y)^2} } \frac{\grad h}{\cQ(y)}}^{=:\widetilde{\cR}_2} , \grad w}_{L^\infty(\Omega),L^1(\Omega)}.
\end{equation}
Some manipulation gives
\begin{equation*}
\begin{split}
    \widetilde{\cR}_2 &= \grad (y+h) \del{ \frac{1}{\cQ(y+h)} - \frac{1}{\cQ(y)} } + \frac{\grad y \grad y^T }{\cQ(y)^2} \frac{\grad h}{\cQ(y)} \\
    &= -\grad (y+h) \del{ \frac{\cQ(y+h) - \cQ(y)}{\cQ(y+h) \cQ(y)} } + \frac{\grad y \grad y^T }{\cQ(y)^2} \frac{\grad h}{\cQ(y)} \\
    &= -\frac{\grad (y+h)}{\cQ(y+h) \cQ(y)} \left( \cR_1 + D_y\cQ(y)\pair{h} \right) + \frac{\grad y \grad y^T }{\cQ(y)^2} \frac{\grad h}{\cQ(y)} \\
    &= \cR_1 O(1) - \frac{\grad (y+h)}{\cQ(y+h) \cQ(y)} \left( \frac{\grad y \cdot \grad h}{\cQ(y)} \right) + \frac{\grad y \grad y^T }{\cQ(y)^2} \frac{\grad h}{\cQ(y)}.
\end{split}
\end{equation*}
Continuing further, we obtain
\begin{equation*}
\begin{split}
  \widetilde{\cR}_2 &= \cR_1 O(1) + O(|\grad h|^2) - \frac{1}{\cQ^2(y)} \left( \grad y \left( \frac{\grad y \cdot \grad h}{\cQ(y+h)} \right) - \grad y \grad y^T \frac{\grad h}{\cQ(y)} \right) \\
  &= \cR_1 O(1) + O(|\grad h|^2) - \frac{\grad y \grad y^T}{\cQ^2(y)} \left( \frac{1}{\cQ(y+h)} - \frac{1}{\cQ(y)} \right) \grad h,
\end{split}
\end{equation*}
and computing the $L^\infty$ norm then yields $\norm{\widetilde{\cR}_2}_{L^\infty(\Omega)} \leq O( \normS{h}1\infty\Omega^2)$, because
\begin{equation}\label{eqn:inverse_Q_Lipschitz}
  \norm{\frac{1}{\cQ(y)}-\frac{1}{\cQ(z)}}_{L^\infty(\Omega)}
     = \norm{\frac{\cQ(z) - \cQ(y)}{\cQ(y) \cQ(z)}}_{L^\infty(\Omega)}
     \le \normS{y-z}1\infty\Omega.
\end{equation}
Combining with \eqref{eqn:Frechet_N_W1inf_pf_1}, we see that $\normS{\cR_2}{-1}\infty\Omega \leq O( \normS{h}1\infty\Omega^2)$ and a standard $\epsilon$-$\delta$ argument proves the Fr\'echet differentiability of $\cN$.  Note that the constants appearing in the above estimates are very mild (most are bounded by 1).

To conclude the proof we need to show the Lipschitz property for $D_y \cN$. Consider a fixed but arbitrary direction $h$, and let $y,z \in v \oplus \sobZ1\infty\Omega$ with $y \neq z$, then
\begin{align*}
&\pair{D_y\cN(y,u) \pair{h} - D_y\cN(z,u) \pair{h},w}_{\sob{-1}\infty\Omega,\sobZ11\Omega}\\
\quad\quad
    &=  \pair{ \del{ \frac{1}{Q(y)}
          - \frac{1}{Q(z)} } \grad h , \grad w}_{L^\infty(\Omega),L^1(\Omega)} - \pair{ \del{ \frac{\grad y \grad y^T}{Q(y)^3} - \frac{\grad z \grad z^T}{Q(z)^3} } \grad h , \grad w}_{L^\infty(\Omega),L^1(\Omega)} \\
    &= I_1 - I_2,
\end{align*}
where $I_1$ is clearly Lipschitz continuous.  Continuing, we have
\begin{align*}
 I_2 &= \pair{\del{\frac{\grad y \del{\grad (y - z)}^T}{Q(y)^3} +
       \frac{ \grad y \del{Q(z)^3 - Q(y)^3} + \del{\grad(y-z)} Q(y)^3}{Q(y)^3Q(z)^3}
       \grad z^T} \grad h , \grad w}_{L^\infty(\Omega),L^1(\Omega)} ,
\end{align*}
and using $a^3 - b^3 = (a-b)(a^2 + a b + b^2)$ and \eqref{eq:Q_Lip} we obtain
\[
    \sup_{h \in \sobZ1\infty\Omega}\frac{\normS{D_y\cN(y,u) \pair{h} - D_y\cN(z,u) \pair{h}}{-1}\infty\Omega}{\normS{h}1\infty\Omega} \le 2 \normS{y-z}1\infty\Omega ,
\]
which completes the proof.
The same argument can be applied to show the twice Fr\'echet differentiability with respect to $y$ with Lipschitz second order derivative (the details are omitted for brevity).
\end{proof}

\subsection{$\sob2p\Omega$-Strong Solution}
We remark that for $p > n$, $\sobZ11\Omega \subset\subset L^{p'}(\Omega)$, consequently $L^p(\Omega) \subset\subset \sob{-1}\infty\Omega$. Recalling that $p > n$ (for a fixed $p$), throughout this section we assume that $v \in \sob2p\Omega$.  We introduce the following space
\[
    Y := \del{v \oplus \sobZ1\infty\Omega}\cap\sob2p\Omega ,
\]
so $y \in Y$ means $y-v \in \sobZ1\infty\Omega \cap \sob2p\Omega$.

\begin{lemma}
\label{lem:Frechet_N_W2p}
Let $U_2 \subset U_1 \cap L^p(\Omega)$ be open, then for every $u \in U_2$ and $v \in \sob2p\Omega$, the operator $\cN(\cdot,u) : Y \rightarrow L^p(\Omega)$ is Fr\'echet differentiable and the Fr\'echet derivative is Lipschitz continuous and is given by
\[
    D_y\cN\del{y,u}\pair{h} = -\divg{ \del{ \del{ \cI - \frac{\grad y \grad y^T }{\cQ(y)^2} } \frac{\grad h}{\cQ(y)}}}.
\]
Moreover $\cN$ is twice Fr\'echet differentiable with Lipschitz second order Fr\'echet derivative.
\end{lemma}
\begin{proof}
For $p>n$, $\sob1p\Omega$ is a Banach algebra. Using this fact the proof is the same as in \lemref{lem:Frechet_N_W1inf}.
\end{proof}
\begin{remark}\label{rem:U_set_nonempty}
We recall that $L^p(\Omega) \subset\subset \sob{-1}\infty\Omega$.  Since $0 \in U_1$, we have that $U_1 \cap L^p(\Omega)$ in \lemref{lem:Frechet_N_W2p} is not empty.  So we can set $U_2 = U_1 \cap L^p(\Omega) \neq \emptyset$.
\end{remark}

Next we will state the existence and uniqueness of $y \in Y$ satisfying \eqref{eq:state}. Remarkably enough, we not only get the improved regularity for $y$ but also the Fr\'echet differentiability of the control to state map (compare with \cite[Section~1.4.2]{MHinze_RPinnau_MUlbrich_SUlbrich_2009a}).  First, we recall the implicit function theorem from \cite[2.7.2]{LNirenberg_1974a}.
\begin{theorem}[implicit function theorem]\label{thm:implicit_func_thm}
Let $X, Y$, and $Z$ be Banach spaces and $f$ a continuous mapping of an open set
$U \subset X \times Y \rightarrow Z$. Assume that $f$ has a Fr\'echet derivative with
respect to $x$, $D_x f(x,y)$, which is continuous in $U$. Let $(x_0,y_0) \in U$ and
$f(x_0,y_0) = 0$. If $D_x f(x_0,y_0)$ is an isomorphism of $X$ onto $Z$ then:
\begin{enumerate}[(i)]
\item There is a ball $B_r(y_0) := \set{y : \norm{y-y_0} < r} \subset Y$ and a unique continuous
      map $g : B_r(y_0) \rightarrow X$ such that $g(y_0) = x_0$ and $f(g(y),y) = 0$, for all $y$ in $B_r(y_0)$.
\item If $f$ is of class $C^1$, then $g(y)$ is of class $C^1$ and
\[
    D_y g(y) = - [D_x f(g(y),y)]^{-1} \circ D_y f(g(y),y) .
\]
\item $D_y g(y)$ belongs to $C^p$ if $f$ is in $C^p$, for $p > 1$.
\end{enumerate}
\end{theorem}

\begin{theorem}[$\sobZ1\infty\Omega \cap \sob2p\Omega$ state]
\label{thm:state_strong}
Let $\Omega$ be $C^{1,1}$ and $v \in \sob2p\Omega$. There exists an open set $U_3 \subset U_2$ such that $0 \in U_3$ and for all $u \in U_3$, there exists a unique solution map $\cS : U_3 \rightarrow Y$ such that
\[
    \cN\del{\cS(u), u} = 0, \quad \mbox{for all } u \in U_3 .
\]
Furthermore, $\cS$ is twice continuously Fr\'echet differentiable as a function of $u$ with first order derivative at $u \in U_3$ given by
\[
    D_u \cS(u) = - \sbr{D_y\cN\del{y,u}}^{-1} \circ D_u \cN\del{y,u} .
\]
\end{theorem}
\begin{proof}
To this end it is sufficient to confirm the hypothesis of \thmref{thm:implicit_func_thm}.
\begin{enumerate}
    \item In view of \lemref{lem:Frechet_N_W2p}, $\cN$ is continuously Fr\'echet differentiable with respect to $y$ on an open subset of $\sob2p\Omega$.

    \item At $(y_0,u_0) = (0,0)$, using \eqref{eq:state_weak} we get $\cN(y_0, u_0) = 0$.

    \item $D_y \cN(y_0, u_0) \pair{h} = -\Delta h$, which is a Banach space isomorphism from $\sob2p\Omega$ to $L^p(\Omega)$ for $\Omega$ of class $C^{1,1}$; see \cite[Theorem~9.15]{DGilbarg_NTrudinger_2001a}.
\end{enumerate}
Using the implicit function theorem, we conclude.
\end{proof}

\subsection{$W^2_p$-Continuity Estimate}
\thmref{thm:state_strong} provides existence and uniqueness of the $\sob2p\Omega$ solution to the state equation but not the continuity estimate for the solution variable. Later we see that the continuity estimate is a crucial piece of the puzzle. We develop a fixed point argument to show the existence and uniqueness in a ball where this a priori estimate holds. The proof requires the boundary data $v \in \sob2p\Omega$ to be small and $u$ to be in an open subset of $L^p(\Omega)$ (see Definition \ref{def:admissible_control}).  We remark that no smallness condition on $v$ was needed previously in \thmref{thm:state_strong}.


We begin by defining a solution set
\begin{equation}\label{eq:B}
    \mathbb{B} = \set{y \in Y \ : \
                      \normS{y}2p\Omega \le B_1},
\end{equation}
with $B_1 > 1$.
For a given $y \in \mathbb{B}$, define a map $T : Y \rightarrow Y$ such that $T(y) = \tilde{y}$ solves
\begin{align}\label{eq:state_linearized}
   - \del{\cQ(y)^2 I - \grad y \grad y^T} : \Dif^2 \tilde{y}
   &= u \cQ(y)^3   \quad \mbox{in } \Omega .
\end{align}
This is a linearization of the state equation \eqref{eq:state} obtained by expanding the left-hand-side of \eqref{eq:state} and evaluating the non-linear ``coefficient'' at $y \in \mathbb{B}$.  

\begin{lemma}\label{eq:positive_coeff}
    The coefficient matrix $\del{\cQ(y)^2 I - \grad y \grad y^T}$ in \eqref{eq:state_linearized} is uniformly positive definite.
\end{lemma}
\begin{proof}
Let ${\bf b} \in \real^n$ be an arbitrary nonzero column vector with components $b_1, \dots, b_n$ and set $E = \cQ(y)^2 I - \grad y \grad y^T$.  Then, using the definition of $\cQ$, we obtain
\begin{align*}
    {\bf b}^T E {\bf b}
   &= {\bf b}^T {\bf b} + \del{ \grad y^T \grad y} \del{{\bf b}^T {\bf b}} - {\bf b}^T \grad y \grad y^T {\bf b} \\
   &= {\bf b}^T {\bf b} + \del{\grad y^T\grad y} \del{{\bf b}^T {\bf b}} - \del{\grad y^T {\bf b}}^T
                    \del{\grad y^T {\bf b}}  \\
   &= {\bf b}^T {\bf b} + \del{\sum_{i=1}^n \abs{\partial_i y}^2} \del{\sum_{j=1}^n \abs{b_j}^2}
     - \abs{\sum_{i=1}^n\partial_i y b_i}^2 \ge {\bf b}^T {\bf b} > 0,
\end{align*}
where we used the Cauchy-Schwarz inequality.
\end{proof}
\begin{lemma}[range of $T$]
There exist constants $C_\Omega > 0$, and $B_2(n,p,B_1,\Omega) > 0$, such that if $v \in \sob2p\Omega$ and $u \in L^p(\Omega)$ satisfy
\begin{equation}\label{eq:cond_T_map_B_to_B}
    C_\Omega \del{\normS{v}2p\Omega + B_2 \norm{u}_{L^p(\Omega)} } \le B_1,
\end{equation}
then $T$ maps $\mathbb{B}$ to $\mathbb{B}$.
\end{lemma}
\begin{proof}
For a given $y \in \mathbb{B}$, $\cQ(y) \in \sob1\infty\Omega$, whence the right hand side in \eqref{eq:state_linearized} belongs to $L^p(\Omega)$. In view of \cite[Theorem~9.15]{DGilbarg_NTrudinger_2001a} in conjunction with \lemref{eq:positive_coeff}, there exists a unique $\tilde{y}$ solving \eqref{eq:state_linearized}. Moreover \cite[Lemma~9.17]{DGilbarg_NTrudinger_2001a} implies there exists a constant $C_\Omega$ such that $\tilde{y}$ satisfies the a priori estimate:
%
\begin{equation*}
    \normS{\tilde{y}}2p\Omega
   \le C_\Omega \del{\normS{v}2p\Omega + \norm{u}_{L^p(\Omega)} \norm{\cQ(y)^3}_{L^\infty(\Omega)} } .
\end{equation*}
Since $y \in \mathbb{B}$ and $\sob2p\Omega \subset\subset \sob1\infty\Omega$ for $p>n$ with embedding constant $C_S$ we deduce
\begin{equation}\label{eq:continuity_estimate_ytl}
    \normS{\tilde{y}}2p\Omega
   \le C_\Omega \del{\normS{v}2p\Omega + B_2 \norm{u}_{L^p(\Omega)} }
\end{equation}
where the constant $B_2$ depends on $B_1$, $p$, $n$ and the embedding constant $C_S$. Choosing $\normS{v}2p\Omega$ and $\norm{u}_{L^p(\Omega)}$ such that \eqref{eq:cond_T_map_B_to_B} hold, we conclude that $T$ maps $\mathbb{B}$ to $\mathbb{B}$.
\end{proof}
\begin{theorem}[fixed point]\label{thm:fixed_point}
If, in addition to \eqref{eq:cond_T_map_B_to_B}, $u \in L^p(\Omega)$ and $v \in \sob2p\Omega$ further satisfy
\begin{equation}\label{eq:cond_T_contraction}
    B_3 \del{2\norm{u}_{L^p(\Omega)}+\normS{v}2p\Omega} < 1,
\end{equation}
for some constant $B_3(n,p,B_1,B_2,\Omega) > 0$
then the map $T : \mathbb{B} \rightarrow \mathbb{B}$ is a contraction.
Moreover, the solution $y$ to the state equation satisfies
\begin{equation}\label{eq:estimate_y_2}
C_S \normS{y}2p\Omega \le \frac{1}{B_1 (1+ C_S^2 B_1^2)^{1/2}},
\end{equation}
\end{theorem}
\begin{proof}
Take $y_1$, $y_2$ in $\mathbb{B}$, with $y_1 \neq y_2$, and let $\tilde{y}_i = T(y_i)$ (for $i=1,2$) solve the linearized system \eqref{eq:state_linearized}.  Define $\delta y := y_1 - y_2$ and $\delta \tilde{y} := \tilde{y}_1 - \tilde{y}_2$.  Computing the difference between the equations satisfied by $\tilde{y}_1$ and $\tilde{y}_2$ and after various algebraic manipulations we deduce
\begin{align*}
    - \del{\cQ(y_2)^2 I - \grad y_2 \grad y_2^T} : \Dif^2 \delta \tilde{y} &= u \del{\cQ(y_1)^3 - \cQ(y_2)^3} \\
   &\quad - \left\{ \cQ(y_2)^2 I - \cQ(y_1)^2 I + \grad \delta y \grad y_1^T + \grad y_2 \grad \delta y^T \right\} :\Dif^2 \tilde{y}_1 .
\end{align*}
Again using the Sobolev embedding theorem, and $p > n$, it is easy to check that the right-hand-side belongs to $L^p(\Omega)$. Toward this end, we invoke \cite[Theorem~9.15]{DGilbarg_NTrudinger_2001a} in conjunction with \lemref{eq:positive_coeff} and \cite[Lemma~9.17]{DGilbarg_NTrudinger_2001a}, and find there exists a constant $C_\Omega$ large enough satisfying $C_\Omega^2 > \frac{B_1 (1+ C_S^2 B_1^2)^{1/2}}{\normS{v}2p\Omega + B_2 \norm{u}_{L^p(\Omega)}}$, such that
\begin{align*}
    \normS{\delta\tilde{y}}2p\Omega
   &\le C_\Omega \bigg( \norm{u \del{\cQ(y_1)^3 - \cQ(y_2)^3}}_{L^p(\Omega)}
    + \norm{\del{ \cQ(y_2)^2 - \cQ(y_1)^2} \Delta \tilde{y}_1}_{L^p(\Omega)} \\
   &\quad + \norm{\grad \delta y \grad y_1^T : \Dif^2 \tilde{y}_1}_{L^p(\Omega)}
    + \norm{\grad y_2 \grad \delta y^T : \Dif^2 \tilde{y}_1 }_{L^p(\Omega)} \bigg) .
\end{align*}
We further deduce
\begin{align*}
     \normS{\delta\tilde{y}}2p\Omega
   &\le C_\Omega \bigg( \norm{u}_{L^p(\Omega)} \norm{\cQ(y_1)^3 - \cQ(y_2)^3}_{L^\infty(\Omega)} + \norm{\cQ(y_1)^2 - \cQ(y_2)^2}_{L^\infty(\Omega)} \normS{\tilde{y}_1}2p\Omega \nonumber \\
   &\qquad\quad
    + \normSZ{y_1}1\infty\Omega \normS{\tilde{y}_1}2p\Omega \normSZ{\delta y}1\infty\Omega + \normSZ{y_2}1\infty\Omega \normS{\tilde{y}_1}2p\Omega \normSZ{\delta y}1\infty\Omega
    \bigg) \nonumber \\
   &= \textsf{I} + \textsf{II} + \textsf{III} + \textsf{IV}.
\end{align*}
Regarding the terms $\textsf{III}$ and $\textsf{IV}$, it suffices to estimate $\textsf{III}$. For every $h \in \sob2p\Omega$ we have $\normSZ{h}1\infty\Omega \le C_S \normS{h}2p\Omega$, and $\tilde{y}_1$ satisfies \eqref{eq:continuity_estimate_ytl}, therefore
\[
\textsf{III} \le C_\Omega^2 C_S \normS{y_1}2p\Omega \del{\normS{v}2p\Omega + B_2 \norm{u}_{L^p(\Omega)}}  \normSZ{\delta y}1\infty\Omega .
\]
To estimate $\textsf{I}$ and $\textsf{II}$, we use the fact that $\cQ$ is Lipschitz continuous (see the proof of \lemref{lem:Frechet_N_W1inf}), $\tilde{y}$ satisfies \eqref{eq:continuity_estimate_ytl}, and $y_1,y_2 \in \mathbb{B}$, to obtain
%
\begin{align*}
     \normS{\delta\tilde{y}}2p\Omega
         \le B_3 \del{\norm{u}_{L^p(\Omega)}+\normS{v}2p\Omega} \normS{\delta y}2p\Omega,
\end{align*}
where the constant $B_3$ depends on $C_\Omega$, $B_1$, $p$, and $C_S$ where the latter is the embedding constant for $\sob2p\Omega$ in $\sob1\infty\Omega$ for $p>n$. Choosing $u$ and $v$ such that \eqref{eq:cond_T_contraction} holds, we get the desired contraction.

Clearly, a necessary condition (see \textsf{III}) for \eqref{eq:cond_T_contraction} to hold is
\[
C_\Omega^2 C_S \normS{y_1}2p\Omega \del{\normS{v}2p\Omega + B_2 \norm{u}_{L^p(\Omega)}}
< 1 .
\]
The choice of the constant $C_\Omega$ leads to \eqref{eq:estimate_y_2}.
\end{proof}
%


%
\begin{definition}[control sets $U$ and $U_{ad}$]
\label{def:admissible_control}
    Recall that $p > n$ is fixed.  We define an open set
    \[
        U := \set{u \in L^p(\Omega) : \eqref{eq:cond_T_map_B_to_B} \mbox{ and } \eqref{eq:cond_T_contraction} \mbox{ holds} } \cap U_3.
    \]
    Next, define the closed set of admissible controls
    \[
        U_{ad} := \set{u \in L^2(\Omega) : \norm{u}_{L^p(\Omega)} \le \theta, \ p > n  } ,
    \]
    where $\theta$ is chosen such that $U_{ad} \subset U$.  The set $U_{ad}$ is nonempty (see Remark \ref{rem:U_set_nonempty}).
\end{definition}
\begin{lemma}[$S$ Lipschitz]
\label{lem:S_Lipschitz}
Recall that $\cS$ is the control to state map.  If $u_1, u_2 \in U$, then
\begin{align}
    \normS{\cS(u_1)-\cS(u_2)}12\Omega &\le C(n,p,B_1,\Omega) \normS{u_1-u_2}{-1}2\Omega, \label{eq:S_Lipschitz} \\
    \normS{\cS(u_1)-\cS(u_2)}2p\Omega &\le C(n,p,B_1,\Omega) \norm{u_1-u_2}_{L^p(\Omega)}. \label{eq:S_Lipschitz_pt2}
\end{align}
\end{lemma}
\begin{proof}
Recall the equations satisfied by $\cS(u_1) \in Y$ and $\cS(u_2) \in Y$
\begin{align*}
    -\divg{\del{\frac{1}{\cQ(\cS(u_1))}\grad \cS(u_1)}} = u_1 , \qquad
    -\divg{\del{\frac{1}{\cQ(\cS(u_2))}\grad \cS(u_2)}} = u_2 .
\end{align*}
On subtracting and rearranging, we obtain
\[
    -\divg{\del{\frac{1}{\cQ(\cS(u_1))}\grad \del{\cS(u_1)-\cS(u_2)}}}
        = \divg{\del{\del{\frac{1}{\cQ(\cS(u_1))}-\frac{1}{\cQ(\cS(u_2))}} \grad \cS(u_2)}} +  u_1 - u_2    .
\]
Using the characterization of $\sob{-1}2\Omega$ functions \cite[P.~283, Theorem 1]{LCEvans_1998a}, and $\cS(u_2) \in \mathbb{B}$ which implies $\cQ(\cS(u_1)) \le \del{1+ C_S^2 B_1^2}^{1/2}$, we have the a priori estimate of the solution $\cS(u_1)-\cS(u_2)$ of the above elliptic PDE:
\begin{align*}
    \normSZ{\cS(u_1)-\cS(u_2)}12\Omega
   &\le \del{1+ C_S^2 B_1^2}^{1/2}  \normSZ{\cS(u_2)}1\infty\Omega
        \norm{\frac{1}{\cQ(\cS(u_1))}-\frac{1}{\cQ(\cS(u_2))}}_{L^2(\Omega)}
              + \normS{u_1 - u_2}{-1}2\Omega   .
\end{align*}
Moreover, for $p>n$, $\sob2p\Omega \subset\subset \sob1\infty\Omega$ with embedding constant $C_S$, we get
\begin{align*}
    \normSZ{\cS(u_1)-\cS(u_2)}12\Omega
   &\le C_S \normS{\cS(u_2)}2p\Omega \del{1+ C_S^2 B_1^2}^{1/2}
            \normSZ{\cS(u_1)-\cS(u_2)}12\Omega
          + \normS{u_1 - u_2}{-1}2\Omega .
\end{align*}
Finally due to \eqref{eq:estimate_y_2} we have
\[
    C_S \normS{\cS(u_2)}2p\Omega \del{1+ C_S^2 B_1^2}^{1/2}
        < \frac{1}{B_1} < 1 ,
\]
where the last inequality is due to the fact that $B_1 > 1$, which implies \eqref{eq:S_Lipschitz}.

To prove \eqref{eq:S_Lipschitz_pt2}, set $y_1 = \cS(u_1)$ and $y_2 = \cS(u_2)$, and rewrite the state equation \eqref{eq:state} into the non-divergence form as
\[
- \del{\cQ(y_1)^2 I - \grad y_1 \grad y_1^T} : \Dif^2 y_1
   = u_1 \cQ(y_1)^3   ~ \mbox{ in } \Omega,
   \quad \mbox{and } ~
- \del{\cQ(y_2)^2 I - \grad y_2 \grad y_2^T} : \Dif^2 y_2
   = u_2 \cQ(y_2)^3   ~ \mbox{ in } \Omega.
\]
After subtracting, rearranging, and setting $\delta y = y_1 - y_2$, we obtain
\begin{align*}
    - \del{\cQ(y_2)^2 I - \grad y_2 \grad y_2^T} : \Dif^2 \delta y
   &= (u_1 - u_2) \cQ(y_1)^3 + u_2 \del{\cQ(y_1)^3 - \cQ(y_2)^3} \\
   &\qquad   - \Big{\{} \cQ(y_2)^2 I - \cQ(y_1)^2 I
   + \grad \delta y \grad y_1^T + \grad y_2 \grad \delta y^T \Big{\}} :\Dif^2 y_1 .
\end{align*}
Recalling \eqref{eq:cond_T_contraction}, the remaining proof is similar to the proof of \thmref{thm:fixed_point}
and is omitted to avoid repetition.
\end{proof}

\section{Optimality Conditions}
\label{s:reduced_problem}
Using the control to state map, we can rewrite the minimization problem \eqref{eq:cost}-\eqref{eq:state} in the following reduced form:
\begin{align}
\label{eq:reduced_cost}
    \inf \cJ(u) := \cJ(\cS(u),u) \quad \mbox{over } u \in U_{ad},
\end{align}
where
\[
    \cJ(\cS(u),u) = \cJ_1(\cS(u)) + \cJ_2(u),
\]
with
\[
    \cJ_1(\cS(u)) = \frac{1}{2}\normLt{S(u)-y_d}\Omega^2, \quad
    \cJ_2(u) = \frac{\alpha}{2}\normLt{u}\Omega^2.
\]

We begin by introducing the notion of a minimizer for our optimal control problem.
\begin{definition}[optimal control]
\label{def:optimal_control}
A control $\uop \in U_{ad}$ is said to be \emph{optimal} if it satisfies, together with the associated optimal state $\yop(\uop) := \cS(\uop)$,
\[
\cJ(y(u),u) \ge \cJ(\yop(\uop), \uop) \quad \mbox{for all } u \in U_{ad} .
\]
A control $\uop \in U_{ad}$ is said to be \emph{locally optimal} in the sense of $L^p(\Omega)$, if there exists an
$\epsilon > 0$ such that above inequality holds for all $u \in U_{ad}$ such that $\norm{u - \uop}_{L^p(\Omega)} \le \epsilon$.
\end{definition}
The above definition clearly distinguishes between local and global solutions to our optimal control problem. Although in \thmref{thm:control_existence} we prove the existence of a global optimal control, a local optimal control plays a central role in optimization theory and algorithms. Generally speaking, gradient based numerical schemes only guarantee a local optimal solution. Thus, we state our first order necessary optimality conditions in \thmref{thm:first_optimality} in terms of a local optimal control. Existence of such a local optimal control is shown in
Corollary~\ref{cor:quad_growth} under a second order condition (Assumption~\ref{ass:ssc}).
In order to get to \thmref{thm:first_optimality} and Corollary~\ref{cor:quad_growth}, we prove several new results which do not assume the local condition on the control and are central to this paper. In particular, \lemref{lem:regularity_non_div} is a standalone result which further extends the regularity theory of elliptic PDEs in non-divergence form.
Moreover, Proposition~\ref{prop:frechet_derivatives} and Lemmas \ref{lem:A_Lipschitz}, \ref{lem:Sp_Lipschitz} hold for an arbitrary $u \in U$ (recall Definition \ref{def:admissible_control}).

\subsection{Existence Of An Optimal Control}

%
\begin{theorem}
\label{thm:control_existence}
There exists an optimal control $\uop$ solving the reduced minimization problem \eqref{eq:reduced_cost}.
\end{theorem}
\begin{proof}
The proof is based on a minimizing sequence argument. As $\cJ$ is bounded below, there exists a minimizing sequence $\{u_n\}_{n\in\mathbb{N}}$, i.e.
\[
    \inf_{u\in U_{ad}} \cJ(S(u),u) = \lim_{n\rightarrow \infty} \cJ(S(u_n),u_n).
\]
By Definition~\ref{def:admissible_control}, $U_{ad}$ is a nonempty, closed, bounded and convex subset of $L^p(\Omega)$ which is a reflexive Banach space for $n < p < \infty$, thus weakly sequentially compact. Consequently, we can extract a weakly convergent subsequence $\{u_{n_k}\}_{k\in \mathbb{N}} \subset L^p(\Omega)$ i.e.
\[
    u_{n_k} \rightharpoonup \uop \mbox{ in } L^p(\Omega), \quad \uop \in U_{ad} .
\]
This $\uop$ is the candidate for our optimal control.

In the sequel, we drop the index $k$ when extracting subsequences. Using \thmref{thm:fixed_point}, $\cS(u_n) = y_n$ satisfies the state equation \eqref{eq:state}. Since $Y \subset\subset v \oplus \sobZ1\infty\Omega$ for $p > n$, the Rellich-Kondrachov theorem yields a strongly convergent subsequence $\{ y_n \}_{n\in\mathbb{N}} \subset v \oplus \sobZ1\infty\Omega$, i.e.
\[
    y_n \rightarrow \yop \mbox{ in } v \oplus \sobZ1\infty\Omega .
\]
Note that the limit $\yop$ is the state corresponding to the control $\uop$. This results from replacing $\yop$ with $y_n$ in the variational equation \eqref{eq:state_weak} taking the limit and making use of the embedding $L^p(\Omega) \subset \subset \sob{-1}\infty\Omega$.

Finally, using the fact that $\cJ_2(u)$ is continuous in $L^2$ and convex, together with the strong convergence $y_n \rightarrow \yop$ in $L^\infty(\Omega)$, it follows that $\cJ$ is weakly lower semicontinous, whence
\[
    \inf_{u \in U_{ad}} \cJ(u) = \liminf_{n\rightarrow \infty} \del{\cJ_1(\cS(u_n))+ \cJ_2(u_n)} \ge \cJ_1(\cS(\uop)) + \cJ_2(\uop) = \cJ(\uop) .
\]
\end{proof}
%

\subsection{First Order Necessary Conditions}
In the following, let $\uop$ denote the \emph{local} optimal control. We derive
the first order necessary optimality conditions that have to be satisfied by $\uop$ with associated state $\yop$. 
We recall the following result from \cite{FTroltzsch_2010a}.
\begin{lemma}\label{lem:first_order_opt}
Recall that $U_{ad} \subset L^p(\Omega)$ is nonempty and convex, and $\mathcal{J}$ is Fr\'echet differentiable in an open subset of $L^p(\Omega)$ containing $U_{ad}$.  If $\uop \in U_{ad}$ denotes a local optimal control,
then the first order necessary optimality condition satisfied by $\uop$ is
\[
    \pair{\cJ'(\uop), u - \uop}_{L^2(\Omega),L^2(\Omega)} \ge 0, \quad \forall u \in U_{ad} .
\]
\end{lemma}
\begin{theorem}
\label{thm:first_optimality}
If $\uop \in U_{ad}$ denotes a local optimal control, then the first-order optimality conditions are given by the state equation \eqref{eq:state}, the adjoint equation
\begin{equation}\label{eq:adjoint_strong}
    -\divg{\del{ A[\yop]
                \grad \adjop}} = \yop-y_d \quad \mbox{ in } \Omega, \qquad
                  \adjop = 0 \quad \mbox{ on } \bdy\Omega,
\end{equation}
where
\begin{equation}\label{eqn:adjoint_coef_A}
  A[\yop] = \frac{1}{Q(\yop)}\del{ \cI - \frac{\grad \yop \grad \yop^T }{Q(\yop)^2} },
\end{equation}
and the equation for the control
\begin{equation}\label{eq:control_weak}
    \pair{\adjop + \alpha \uop, u-\uop}_{L^2(\Omega),L^2(\Omega)} \ge 0, \quad
        \forall u \in U_{ad} .
\end{equation}
\end{theorem}
\begin{proof}
Using \thmref{thm:state_strong} we can infer that $\cJ$ is Fr\'echet differentiable, and the Fr\'echet derivative of $\cJ$ at $\uop$ in a direction $h \in C(\uop)$ is
\[
    \pair{\cJ'(\uop), h}_{L^{p'}(\Omega),L^p(\Omega)}
        = \pair{\cJ_1'(S(\uop)) , \cS'(\uop) h}_{Y^*,Y}
          + \pair{\cJ_2'(\uop), h}_{L^{p'}(\Omega), L^{p}(\Omega)} ,
\]
whence
\begin{align*}
    \pair{\cJ'(\uop), h}_{L^{p'}(\Omega),L^p(\Omega)}
     &= \pair{\yop - y_d , \cS'(\uop) h}_{L^2(\Omega),L^2(\Omega)}
       + \alpha \pair{\uop, h}_{L^2(\Omega),L^2(\Omega)} \\
     &= \pair{\cS'(\uop)^*(\yop - y_d) , h}_{L^2(\Omega),L^2(\Omega)}
       + \alpha \pair{\uop, h}_{L^2(\Omega),L^2(\Omega)}   .
\end{align*}
Recalling the expression for $\cS'(\uop)$ from \thmref{thm:state_strong} and the fact that $D_u\cN(\yop,\uop) = -\cI$, where we have dropped the dependence of $\cN$ on $v$, we get
\begin{align*}
    \pair{\cJ'(\uop), h}_{L^{p'}(\Omega),L^p(\Omega)}
     &= \pair{\sbr{D_y\cN(\yop, \uop)}^{-*}(\yop - y_d) , h}_{L^2(\Omega),L^2(\Omega)} + \alpha \pair{\uop, h}_{L^2(\Omega),L^2(\Omega)} .
\end{align*}
Setting $\adjop = \sbr{D_y\cN(\cS(\uop), \uop)}^{-*} (\yop - y_d)$, we get \eqref{eq:adjoint_strong}.  Moreover, we see that $\cJ'(\uop) = \adjop + \alpha \uop$ which yields \eqref{eq:control_weak}.  We remark that the pairing $\pair{\cJ'(\uop),h}_{L^{p'}(\Omega),L^p(\Omega)}$ can be simply treated as the $L^2$ pairing.
\end{proof}
\begin{remark}\label{rem:functional_derivative}
In general, $\cJ'(u) = \adjoint(y) + \alpha u$ for an arbitrary $u$ in $U_{ad}$, where $y$ solves \eqref{eq:state} with $u$ as right-hand-side, and $\adjoint(y)$ solves \eqref{eq:adjoint_strong} with right-hand-side given by $y - y_d$.
\end{remark}

Next we will generalize a result from Gilbarg-Trudinger \cite[Theorem 9.15, Lemma 9.17]{DGilbarg_NTrudinger_2001a} where the lower order coefficient is in $L^q(\Omega)$, for $q > n$, instead of being in $L^\infty(\Omega)$.  This result is crucial to prove the necessary regularity for the adjoint equation \eqref{eq:adjoint_strong}.
\begin{lemma}\label{lem:regularity_non_div}
If $A \in L^\infty(\Omega)^{n\times n}$, ${\bf b} \in L^q(\Omega)^n$, $n < q < \infty$, then for all $f \in L^r(\Omega)$ with $1 < r \le q$, there exists a unique $w \in \sob2r\Omega \cap \sobZ1r\Omega$ solving
\begin{align}\label{eq:non_div}
\begin{aligned}
    -A:\Dif^2 w - {\bf b}\cdot \grad w &= f \quad \mbox{in } \Omega \\
        w &= 0 \quad \mbox{on } \bdy\Omega ,
\end{aligned}
\end{align}
with
\begin{align}\label{eq:non_div_estimate}
    \normS{w}2r\Omega \le C_\Omega \norm{f}_{L^r(\Omega)} .
\end{align}
%
\end{lemma}
\begin{proof}
We prove the result in two steps.

\emph{1. Existence and Uniqueness.} As $L^\infty(\Omega)$ is dense in $L^q(\Omega)$, for ${\bf b} \in L^q(\Omega)^n$ there exists $\{ {\bf b}_{m} \}_{m \in \mathbb{N}} \subset L^\infty(\Omega)^n$ such that
$    {\bf b}_m \rightarrow {\bf b} \quad \mbox{in } L^q(\Omega)^n . $
Similarly as $C^\infty(\Omega)$ is dense in $L^r(\Omega)$, therefore there exists $\{ f_m\}_{m \in \mathbb{N}} \subset C^\infty(\Omega)$ such that $f_m \rightarrow f$ in $L^r(\Omega)$. If we consider the auxiliary problem
\begin{align*}
    -A:\Dif^2 w_m - {\bf b}_m\cdot \grad w_m &= f_m \quad \mbox{in } \Omega \\
        w_m &= 0 \quad \mbox{on } \bdy\Omega ,
\end{align*}
using \cite[Lemma 9.17]{DGilbarg_NTrudinger_2001a}, we deduce
\[
    \normS{w_m}2r\Omega \le C_\Omega \norm{f_m}_{L^r(\Omega)}, \quad \forall r \in (1, \infty) ,
\]
and the right hand side converges to $\norm{f}_{L^r(\Omega)}$.
Since a unit ball in $\sob2r\Omega$ is weakly compact, there exists a subsequence, still labeled $w_m$, that converges weakly in $\sob2r\Omega$ and for $s = \frac{rq}{q-r}$ strongly in $\sob1s\Omega$ to a function $w \in \sob2r\Omega \cap \sobZ1r\Omega$.
It remains to show that $w$ satisfies \eqref{eq:non_div}. Because
\[
    \abs{\int_\Omega v\del{ {\bf b}_m \cdot \grad w_m}}
        \le \norm{v}_{L^{r'}(\Omega)} \norm{{\bf b}_m}_{L^{q}(\Omega)} \normS{w_m}1s\Omega ,
\]
we obtain
\[
    \int_\Omega f_m v = - \int_\Omega v \del{A : \Dif^2 w_m + {\bf b}_m \cdot \grad w_m}
        \rightarrow \int_\Omega f v = - \int_\Omega v \del{A : \Dif^2 w + {\bf b} \cdot \grad w} ,
\]
for all $v \in L^{r'}(\Omega)$.

\emph{2. Continuity estimate.} We first rewrite \eqref{eq:non_div}:
\begin{align*}
    -A:\Dif^2 w &= f + {\bf b}\cdot \grad w \quad \mbox{in } \Omega, \\
        w &= 0 \quad \mbox{on } \bdy\Omega .
\end{align*}
In view of the definition of $s=\frac{rq}{q-r}$, it immediately follows that
$f + {\bf b}\cdot \grad w \in L^r(\Omega)$, whence \cite[Lemma 9.17]{DGilbarg_NTrudinger_2001a} implies
\begin{align}\label{eq:non_div_estimate_1}
    \normS{w}2r\Omega \le C_\Omega \del{ \norm{f}_{L^r(\Omega)}
                          + \norm{\bf b}_{L^q(\Omega)} \normS{w}1s\Omega }  .
\end{align}
Toward this end, we will prove \eqref{eq:non_div_estimate} by contradiction. Let $\{ w_m \}_{m \in \mathbb{N}} \subset \sob2r\Omega \cap \sobZ1r\Omega$ be a sequence satisfying
\[
    \normS{w_m}2r\Omega = 1, \quad \norm{f_m}_{L^r(\Omega)} \rightarrow 0
\]
as $m \rightarrow \infty$, where $f_m = -A:\Dif^2 w_m - {\bf b}\cdot \grad w_m$.
Since the unit ball of $\sob2r\Omega$ is weakly compact, there exists a subsequence, that converges weakly in $\sob2r\Omega$ and strongly in $\sob1s\Omega$ to a $w \in \sob2r\Omega \cap \sobZ1r\Omega$. Therefore,
\[
    \int_\Omega f_m v = - \int_\Omega v \del{A:\Dif^2 w_m + {\bf b}\cdot \grad w_m}
    \rightarrow - \int_\Omega v \del{A:\Dif^2 w + {\bf b}\cdot \grad w} = 0,
\]
for all $v \in L^{r'}(\Omega)$, whence $-A:\Dif^2 w - {\bf b} \cdot \grad w = 0$ and $w = 0$ by uniqueness. But from \eqref{eq:non_div_estimate_1} we deduce
\[
    1 \le C_\Omega \norm{\bf b}_{L^q(\Omega)} \normS{w}1s\Omega ,
\]
which is a contradiction. Thus, \eqref{eq:non_div_estimate} holds.
\end{proof}
\begin{corollary}[regularity of the adjoint]
\label{cor:regularity_adjoint}
For every local optimal control $\uop$, there exists a unique $\adjop \in \sob22\Omega \cap \sobZ12\Omega$. If in addition $y_d \in L^p(\Omega)$, $p > n$, then $\adjop \in \sob2p\Omega$.
\end{corollary}
\begin{proof}
Rewriting \eqref{eq:adjoint_strong}
\begin{equation*}
    -A[\yop]: \Dif^2 \adjop - \divg{\del{A[\yop]}} \cdot \grad \adjop
             = \yop-y_d \quad \mbox{ in } \Omega, \qquad
                  \adjop = 0 \quad \mbox{ on } \bdy\Omega
\end{equation*}
Since $\yop \in \sob2p\Omega$, $p > n$, therefore $A[\yop] \in \sob1p\Omega$, and $\divg{\del{A[\yop]}} \in L^p(\Omega)$, then invoking \lemref{lem:regularity_non_div}, with $q=p$, we obtain the desired result.
\end{proof}
\begin{corollary}[regularity of the optimal control]
\label{cor:regularity_control}
Let $\uop$ denote a local optimal control. In view of \eqref{eq:control_weak} we have
\begin{equation*}
\begin{split}
  \uop &= -\frac{\adjop}{\alpha}, \qquad\qquad \text{if }~ \adjop + \alpha \uop = 0, \\
  \uop &= -\frac{\theta\adjop}{\norm{\adjop}_{L^p(\Omega)}}, \quad \text{if }~ \adjop + \alpha \uop \neq 0.
\end{split}
\end{equation*}
Then invoking Corollary~\ref{cor:regularity_adjoint} and the Sobolev embedding theorem we deduce that $\uop \in W^2_2(\Omega)$ and further if $y_d \in L^p(\Omega)$, $p>n$, then $\uop \in \sob2p\Omega \subset\subset \sob1\infty\Omega$.
\end{corollary}
%

\subsection{Second Order Sufficient Conditions}

We investigate the second order behavior of the cost functional $\cJ$.  Starting from Assumption \ref{ass:ssc}, we build up several intermediate results that allow us to prove Corollary \ref{cor:quad_growth} which is a quadratic growth condition on $\cJ$ near the optimal solution $\uop$.
In order to carefully handle the $L^2$-$L^p$ norm discrepancy, we prove a Lipschitz continuity type result for $\cJ''$ in \lemref{lem:aux_estimate}.  This requires several intermediate results which are shown in Proposition~\ref{prop:frechet_derivatives}, \lemref{lem:A_Lipschitz} and \lemref{lem:Sp_Lipschitz}.

Since $U_{ad}$ is closed, we need to define a suitable set of admissible directions.
\begin{definition}\label{defn:convex_cone}
Given $u \in U_{ad}$, the convex cone $\cone{u}$ comprises all directions $h \in L^p(\Omega)$ such that $u + th \in U_{ad}$ for some $t > 0$, i.e.
\[
    \cone{u} := \set{ h \in L^p(\Omega) : u + th \in U_{ad}, \ \text{ for some } \ t > 0 } .
\]
\end{definition}

\begin{assumption}\label{ass:ssc}
We make the following standard assumption about the second order behavior of the cost functional:
\begin{equation}\label{eq:second_order_sufficient_condition}
    \cJ''(\uop) (u - \uop)^2 \ge \delta \norm{u-\uop}_{L^2(\Omega)}^2, \quad \forall u - \uop \in \cone{\uop}, \quad \text{for some fixed } \delta > 0.
\end{equation}
\end{assumption}

Our next goal is to prove the following crucial result:
\begin{corollary}[quadratic growth near a local optimal control]\label{cor:quad_growth}
Let the control $\uop \in U_{ad}$ satisfy the first order necessary optimality condition \eqref{eq:control_weak} and assume that \eqref{eq:second_order_sufficient_condition} holds. Then there exists an $\epsilon > 0$ such that, for all $u - \uop \in \cone{\uop}$ with $\norm{u-\uop}_{L^p(\Omega)} \le \epsilon$, we have
\begin{equation}  \label{eq:gradJ_quad_growth}
\pairLt{\cJ'(u) - \cJ'(\uop), u-\uop}{\Omega} \ge \frac{\delta}{2} \normLt{u - \uop}{\Omega}^2 ,
\end{equation}
and
\begin{equation}\label{eq:J_quad_growth}
\cJ(u) \ge \cJ(\uop) + \frac{\delta}{4} \norm{u - \uop}_{L^2(\Omega)}^2 .
\end{equation}
In particular, $\cJ$ has a local optimal control (see Definition~\ref{def:optimal_control}) in $U_{ad}$ at $\uop$.
\end{corollary}
The proof requires a non-trivial estimate which we will prove in \lemref{lem:aux_estimate}. Such an estimate is needed to deal with the so-called \emph{2-norm discrepancy}, we refer to \cite{ECasas_FTroeltzsch_2012a} for further reading on the subject. We will conclude this section with a proof of Corollary~\ref{cor:quad_growth}.

\begin{proposition}\label{prop:frechet_derivatives}
For every $u \in U$ and every $h_1, h_2 \in L^p(\Omega)$ the first and second order
Fr\'echet derivatives $\cS'(u)h_1 \in \sob2p\Omega$ and $\cS''(u)h_1 h_2 \in \sob2p\Omega$ at $\cS(u) \in Y$ satisfy
\begin{align}
    -\divg{\del{ A[\cS(u)] \grad \cS'(u)h_1}} &= h_1 \quad \mbox{ in } \Omega, \qquad
                     \cS'(u)h_1 = 0 \quad \mbox{ on } \bdy\Omega
                     \label{eq:first_Frechet_derivative} \\
    -\divg{\del{ A[\cS(u)] \grad \cS''(u)h_1 h_2}} &= \divg{ \del{ D_u A[\cS(u)] \pair{h_2} \grad \cS'(u)h_1 }} \quad \mbox{ in } \Omega, \qquad
                     \cS''(u)h_1 h_2 = 0 \quad \mbox{ on } \bdy\Omega
                     \label{eq:second_Frechet_derivative}
\end{align}
where $A[\cdot]$ is given in \eqref{eqn:adjoint_coef_A}, and
\begin{align}
    \normS{\cS'(u)h_1}12\Omega &\le C(n,p,\Omega) \normS{h_1}{-1}2\Omega , \qquad
    \normS{\cS'(u)h_1}2p\Omega \le  C(n,p,\Omega) \norm{h_1}_{L^p(\Omega)}
     \label{eq:first_Frechet_derivative_estimate} \\
    \normS{\cS''(u)h_1 h_2}12\Omega &\le C(n,p,B_1,\Omega) \normS{h_1}{-1}2\Omega \norm{h_2}_{L^p(\Omega)} \label{eq:second_Frechet_derivative_estimate}.
\end{align}
\end{proposition}
\begin{proof}
In terms of the control to state map, \eqref{eq:state} can be written as $-\divg{\frac{\grad \cS(u)}{\cQ(\cS(u))}} = u$.
Since the control to state map is twice Fr\'echet differentiable, then differentiating with respect to $u$ in the directions $h_1$ and $h_2$ leads to \eqref{eq:first_Frechet_derivative} and \eqref{eq:second_Frechet_derivative}. 
The first inequality in \eqref{eq:first_Frechet_derivative_estimate} is due to the characterization of $\sob{-1}2\Omega$ functions \cite[P.~283, Theorem 1]{LCEvans_1998a} and the second inequality is due to \lemref{lem:regularity_non_div}.  Using both of these results, in conjunction with the Sobolev embedding $\sob2p\Omega \subset\subset \sob1\infty\Omega$ for $p>n$, gives \eqref{eq:second_Frechet_derivative_estimate}.
\end{proof}
\begin{lemma}[$A$ is Lipschitz]\label{lem:A_Lipschitz}
If $u_1, u_2 \in U$, with $u_1 \neq u_2$, the map $A : Y \rightarrow v \oplus \sobZ1p\Omega$ in \eqref{eq:adjoint_strong} satisfies
\begin{align}
    \norm{A[\cS(u_1)] - A[\cS(u_2)]}_{L^\infty(\Omega)} &\le C(n,p,B_1,\Omega) \normS{\cS(u_1) - \cS(u_2)}1\infty\Omega \label{eq:A_Lipschitz_Linf} , \\
    \normLt{A[\cS(u_1)] - A[\cS(u_2)]}\Omega &\le C(n,p,B_1,\Omega) \normS{\cS(u_1)-\cS(u_2)}12\Omega \label{eq:A_Lipschitz_L2} ,
\end{align}
and for $h_1 \in L^p(\Omega)$, $\cS' : U \rightarrow \cL\del{L^p(\Omega),Y}$:
\begin{equation}\label{eq:DuA_Lipschitz_L2}
 \norm{D_u\del{A[\cS(u_1)] - A[\cS(u_2)]}\pair{h_1}}_{L^2(\Omega)}
         \le C(n,p,B_1,\Omega) \normS{\del{\cS'(u_1)-\cS'(u_2)}h_1}12\Omega .
\end{equation}
\end{lemma}
\begin{proof}
Recall $y_1 = \cS(u_1)$ and $y_2 = \cS(u_2)$, for simplicity we will use this notation in the proof. It is enough to show \eqref{eq:A_Lipschitz_Linf}, the same proof works for \eqref{eq:A_Lipschitz_L2} and \eqref{eq:DuA_Lipschitz_L2}. Now
\[
    \norm{A[y_1] - A[y_2]}_{L^\infty(\Omega)}
     \le \norm{\frac{1}{\cQ(y_1)}-\frac{1}{\cQ(y_2)}}_{L^\infty(\Omega)}
      + \norm{ \frac{\grad y_1 \grad y^T_1}{\cQ(y_1)^3}
         - \frac{\grad y_2 \grad y^T_2}{\cQ(y_2)^3}}_{L^\infty(\Omega)} .
\]
We consider each term on the right hand side separately.  For the first term, we recall \eqref{eqn:inverse_Q_Lipschitz}.  Invoking the triangle inequality on the second term leads to
\begin{align*}
    \norm{ \frac{\grad y_1 \grad y^T_1}{\cQ(y_1)^3} - \frac{\grad y_2 \grad y^T_2}{\cQ(y_2)^3}}_{L^\infty(\Omega)}
    &\le \norm{\frac{\grad y_1 \grad y^T_1 \cQ(y_2)^3-\grad y_2 \grad y_2^T \cQ(y_1)^3}{Q(y_1)^3 \cQ(y_2)^3}}_{L^\infty(\Omega)} \\
    &\le \norm{\frac{\grad (y_1 - y_2)\grad y^T_1 \cQ(y_2)^3}{\cQ(y_1)^3 \cQ(y_2)^3}}_{L^\infty(\Omega)}
       + \norm{\frac{\grad y_2 \grad(y_1-y_2) \cQ(y_2)^3}{Q(y_1)^3 \cQ(y_2)^3}}_{L^\infty(\Omega)} \\
    &\quad + \norm{\frac{\grad y_2 \grad y_2^T (\cQ(y_2)^3 - \cQ(y_1)^3)}{\cQ(y_1)^3Q(y_2)^3}}_{L^\infty (\Omega)} \\
    &\le C \abs{y_1-y_2}_{\sob1\infty\Omega} ,
\end{align*}
where $C > 0$ is a generic uniform constant depending on $n,p,\Omega$ and $B_1$.
\end{proof}
\begin{lemma}[$\cS'$ is Lipschitz]\label{lem:Sp_Lipschitz}
Let $u, u_1, u_2 \in U$, and $h_1 \in L^p(\Omega)$. Then
$\cS' : U \rightarrow \cL\del{L^p(\Omega),Y}$ satisfies
\begin{align}\label{eq:Sp_Lipschitz}
    \normS{\del{\cS'(u_1) - \cS'(u_2)}h_1}12\Omega \le \norm{u_1-u_2}_{L^p(\Omega)} \normLt{h_1}\Omega .
\end{align}
\end{lemma}
\begin{proof}
Consider the system satisfied by $\cS'(u_1)h_1$ and $\cS'(u_2)h_1$ from Proposition \ref{prop:frechet_derivatives}:
\begin{align*}
    -\divg{\del{ A[\cS(u_1)] \grad \cS'(u_1)h_1}} &= h_1 \quad \mbox{ in } \Omega, \qquad
                     \cS'(u_1)h_1 = 0 \quad \mbox{ on } \bdy\Omega \\
    -\divg{\del{ A[\cS(u_2)] \grad \cS'(u_2)h_1}} &= h_1 \quad \mbox{ in } \Omega, \qquad
                     \cS'(u_2)h_1 = 0 \quad \mbox{ on } \bdy\Omega
\end{align*}
On subtracting and rearranging
\begin{align*}
    -\divg{\del{ A[\cS(u_1)] \grad \del{\cS'(u_1) - \cS'(u_2)}h_1}}
   &= \divg{\del{A[\cS(u_1)]-A[\cS(u_2)] \grad \cS'(u_2)h_1 }}
               \quad \mbox{ in } \Omega \\
    \del{\cS'(u_1) - \cS'(u_2)} h_1
   &= 0 \quad \mbox{ on } \bdy\Omega  .
\end{align*}
Using the characterization of $\sob{-1}2\Omega$ functions \cite[P.~283, Theorem 1]{LCEvans_1998a} we deduce
\[
    \normS{\del{\cS'(u_1) - \cS'(u_2)}h_1}12\Omega
                \le C(\Omega) \norm{A[\cS(u_1)]-A[\cS(u_2)]}_{L^\infty(\Omega)}
                        \normS{\cS'(u_2)h_1}12\Omega .
\]
Using \eqref{eq:A_Lipschitz_Linf} and \eqref{eq:first_Frechet_derivative_estimate}, we obtain
\[
    \normS{\del{\cS'(u_1) - \cS'(u_2)}h_1}12\Omega
        \le C(n,p,B_1,\Omega) \normS{\cS(u_1)-\cS(u_2)}1\infty\Omega \normS{h_1}{-1}2\Omega .
\]
Using \eqref{eq:S_Lipschitz_pt2} and $\sob{-1}2\Omega \hookrightarrow L^2(\Omega)$ we get \eqref{eq:Sp_Lipschitz}.
\end{proof}

The treatment of the $L^2$-$L^p$ norm discrepancy requires a technical result.
This result makes use of the previous estimates in this section.
\begin{lemma}[auxiliary result for the $L^2$-$L^p$ norm discrepancy]
\label{lem:aux_estimate}
Let $u \in U$ and $y_d, h, h_1, h_2 \in L^p(\Omega)$.  Then there exists a constant
$L(n,p,B_1,\Omega)  > 0$ such that
\begin{align}\label{eq:auxiliary_estimate}
    \abs{\cJ''(u+h)\pair{h_1,h_2} - \cJ''(u)\pair{h_1,h_2}}
        \le L \del{ \norm{h}_{L^2(\Omega)}  \norm{h_2}_{L^p(\Omega)}
                  + \norm{h}_{L^p(\Omega)}  \norm{h_2}_{L^2(\Omega)} } \norm{h_1}_{L^2(\Omega)} .
\end{align}
\end{lemma}
\begin{proof}
Using the reduced cost functional \eqref{eq:reduced_cost}, a simple calculation gives
\begin{align*}
    \cJ''(u+h)\pair{h_1,h_2} - \cJ''(u)\pair{h_1,h_2}
     &= \int_\Omega \del{\cS'(u+h)^2 - \cS'(u)^2} h_1 h_2 \\
     &\quad +\int_\Omega \sbr{\del{\cS(u+h)-y_d} \cS''(u+h)
                 - \del{\cS(u)-y_d} \cS''(u)}h_1 h_2 \\
     &= \int_\Omega \del{\cS'(u+h) - \cS'(u)} h_1 \del{\cS'(u+h) + \cS'(u)}  h_2  \\
     &\quad  +\int_\Omega \sbr{\del{\cS(u+h)-\cS(u)} \cS''(u+h)
             +  (\cS(u) - y_d) \del{\cS''(u+h) - \cS''(u)} } h_1 h_2.
\end{align*}
Using the triangle inequality and Cauchy-Schwarz, we have
\begin{align*}
    \abs{\cJ''(u+h)\pair{h_1,h_2} - \cJ''(u)\pair{h_1,h_2}}
        &\le \normLt{\del{\cS'(u+h) - \cS'(u)}h_1}\Omega
             \normLt{\del{\cS'(u+h) + \cS'(u)}h_2}\Omega   \\
        &\quad + \normLt{\cS(u+h)-\cS(u)}\Omega \normLt{\cS''(u+h)h_1 h_2}\Omega  \\
        &\quad + \abs{\int_\Omega (\cS(u) - y_d) \del{\cS''(u+h) - \cS''(u)} h_1 h_2 } \\
        &= \textsf{I} + \textsf{II} + \textsf{III} .
\end{align*}

We will estimate each term $\textsf{I}-\textsf{III}$ individually. In view of \eqref{eq:Sp_Lipschitz}, \eqref{eq:first_Frechet_derivative_estimate}
\[
    \textsf{I} \le C(n,p,B_1,\Omega)  \norm{h}_{L^p(\Omega)} \normLt{h_1}\Omega \normLt{h_2}\Omega
\]
and using \eqref{eq:S_Lipschitz} and \eqref{eq:second_Frechet_derivative_estimate}
\[
    \textsf{II}  \le C(n,p,B_1,\Omega)  \normLt{h}\Omega \normLt{h_1}\Omega \norm{h_2}_{L^p(\Omega)} .
\]
The estimate for $\textsf{III}$ is more involved.  Recall \eqref{eq:second_Frechet_derivative}, namely the system satisfied by $\cS''(u+h)h_1 h_2$ and $\cS''(u)h_1 h_2$:
\begin{align*}
    -\divg{\del{ A[\cS(u+h)] \grad \cS''(u+h)h_1h_2}} &= \divg{\del{D_u A[\cS(u+h)]\pair{h_2}\grad \cS'(u+h)h_1}} ~ \mbox{ in } \Omega, ~
                     \cS''(u+h)h_1h_2 = 0 ~ \mbox{ on } \bdy\Omega,    \\
    -\divg{\del{ A[\cS(u)] \grad \cS''(u)h_1h_2}} &= \divg{\del{D_u A[\cS(u)]\pair{h_2}\grad \cS'(u)h_1}} ~ \mbox{ in } \Omega, ~
                     \cS''(u)h_1h_2 = 0 ~ \mbox{ on } \bdy\Omega.
\end{align*}
On subtracting and rearranging, we obtain
\begin{align*}
    -\divg{\del{ A[\cS(u)] \grad \del{\cS''(u) - \cS''(u+h)}h_1h_2}}
    &= \divg{\del{\del{ A[\cS(u)] - A[\cS(u+h)]} \grad \cS''(u+h)h_1h_2}} \\
    & + \divg{\del{D_u A[\cS(u)]\pair{h_2}\grad \cS'(u)h_1
      - D_u A[\cS(u+h)]\pair{h_2}\grad \cS'(u+h)h_1}} ,
\end{align*}

For $u \in U$, we denote the variable satisfying \eqref{eq:adjoint_strong} by $\adjoint$, with right hand side $\cS(u) - y_d$. We further deduce
\begin{align*}
    \textsf{III} &= \bigg| \int_\Omega \grad \adjoint \cdot \bigg\{ \del{\del{ A[\cS(u)] - A[\cS(u+h)]} \grad \cS''(u+h)h_1h_2} \\
    &\quad + \del{D_u A[\cS(u)]\pair{h_2}\grad \cS'(u)h_1
      - D_u A[\cS(u+h)]\pair{h_2}\grad \cS'(u+h)h_1} \bigg\} \bigg| \\
    &\le \normS{\adjoint}1\infty\Omega \normLt{A[\cS(u)] - A[\cS(u+h)]}\Omega
         \normS{\cS''(u+h)h_1h_2}12\Omega  \\
    &\quad + \normS{\adjoint}1\infty\Omega
             \normLt{D_u A[\cS(u)]\pair{h_2}}\Omega \normS{\del{\cS'(u)-\cS'(u+h)}h_1}12\Omega \\
    &\quad + \normS{\adjoint}1\infty\Omega
             \normLt{D_u \del{A[\cS(u)]- A[\cS(u+h)]}\pair{h_2}}\Omega \normS{\cS'(u+h)h_1}12\Omega  .
\end{align*}
Using \eqref{eq:A_Lipschitz_L2}, \eqref{eq:S_Lipschitz}, \eqref{eq:second_Frechet_derivative_estimate}, \eqref{eq:DuA_Lipschitz_L2}, \eqref{eq:Sp_Lipschitz} and \eqref{eq:first_Frechet_derivative_estimate}, we obtain
\[
    \textsf{III} \le C(n,p,B_1,\Omega)  \normS{\adjoint}1\infty\Omega \del{ \normLt{h}\Omega \norm{h_2}_{L^p(\Omega)} + \norm{h}_{L^p(\Omega)}  \normLt{h_2}\Omega } \normLt{h_1}\Omega .
\]
\end{proof}

\begin{lemma}[Second order behavior in a neighborhood.]\label{lem:second_order_sufficient_condition_nbd}
If $\uop$ satisfies \eqref{eq:second_order_sufficient_condition} then
\begin{equation}\label{eq:second_order_sufficient_condition_nbd}
\cJ''(u) (\uop-u)^2 \ge \frac{\delta}{2} \norm{\uop-u}_{L^2(\Omega)}^2,
\end{equation}
for all $u \in U_{ad}$ with $\norm{u - \uop}_{L^p(\Omega)} < \frac{\delta}{4L}$.  Note: the argument of $\cJ''$ is different from that in \eqref{eq:second_order_sufficient_condition}.
\end{lemma}
\begin{proof}
We begin by rewriting $\cJ''(u) (\uop - u)^2$:
\begin{align*}
\cJ''(u) (\uop - u)^2
&= \cJ''(\uop) (\uop - u)^2 + \del{\cJ''(u) (\uop - u)^2 - \cJ''(\uop) (\uop - u)^2} \\
&\ge \cJ''(\uop) (\uop - u)^2 - \abs{\del{\cJ''(u) (\uop - u)^2 - \cJ''(\uop) (\uop - u)^2} }
= \textsf{I} - \textsf{II}
\end{align*}
Using \eqref{eq:second_order_sufficient_condition}, we obtain $\textsf{I} \ge \delta \norm{\uop - u}_{L^2(\Omega)}^2$.  And invoking \eqref{eq:auxiliary_estimate} yields
\[
\textsf{II} \le L \del{ \norm{u-\uop}_{L^2(\Omega)} \norm{u-\uop}_{L^p(\Omega)}
                   + \norm{u-\uop}_{L^p(\Omega)} \norm{u-\uop}_{L^2(\Omega)} }
                     \norm{u-\uop}_{L^2(\Omega)}  .
\]
Finally, combining the estimates for \textsf{I} and \textsf{II} gives
\[
\cJ''(u) (u-\uop)^2 \ge \delta \norm{\uop-u}_{L^2(\Omega)}^2
      - 2L \norm{\uop - u}_{L^p(\Omega)} \norm{\uop-u}_{L^2(\Omega)}^2.
\]
For $\norm{u-\uop}_{L^p(\Omega)} < \frac{\delta}{4L}$, we obtain \eqref{eq:second_order_sufficient_condition_nbd}.
\end{proof}

We now arrive at the main result of this section.
\begin{proof}[{\bf Proof of Corollary~\ref{cor:quad_growth}}]
We proceed in two steps:

\noindent
\boxed{1}
Let $u \in U_{ad}$ and $\norm{u-\uop}_{L^p(\Omega)} < \frac{\delta}{8L}$. By Taylor's theorem, there is a $t \in (0,1)$ such that
\begin{align*}
\cJ(u) &= \cJ(\uop) + \pair{\cJ'(\uop),u-\uop} + \frac{1}{2} \cJ''\del{t u + (1-t)\uop}(u-\uop)^2 \\
       &= \cJ(\uop) + \pair{\cJ'(\uop),u-\uop} + \frac{1}{2} \cJ''\del{\uop}(u-\uop)^2
          + \frac{1}{2} \del{ \cJ''\del{t u + (1-t)\uop} - \cJ''\del{\uop} }(u-\uop)^2  \\
       &\ge \cJ(\uop) + \pair{\cJ'(\uop),u-\uop} + \frac{\delta}{2}\norm{u-\uop}_{L^2(\Omega)}^2
          - \abs{\frac{1}{2} \del{ \cJ''\del{t u + (1-t)\uop} - \cJ''\del{\uop} }(u-\uop)^2},
\end{align*}
where the last inequality is due to \eqref{eq:second_order_sufficient_condition}.
Next, \eqref{eq:auxiliary_estimate} gives
\begin{align*}
    \cJ(u) \ge \cJ(\uop) + \pair{\cJ'(\uop),u-\uop} +  \frac{\delta}{2} \normLt{u-\uop}\Omega^2 - 2L \norm{u-\uop}_{L^p(\Omega)}
                                \normLt{u-\uop}\Omega^2,
\end{align*}
which implies
\begin{equation}\label{eq:quad_growth_1}
\cJ(u) \ge \cJ(\uop) + \pair{\cJ'(\uop),u-\uop} + \frac{\delta}{4} \norm{u-\uop}_{L^2(\Omega)}^2 .
\end{equation}
Using \lemref{lem:first_order_opt}, we obtain \eqref{eq:J_quad_growth}.

\noindent
\boxed{2}
Since $\norm{u-\uop}_{L^p(\Omega)} < \frac{\delta}{8L}$ (i.e. $u$ satisfies \eqref{eq:second_order_sufficient_condition_nbd}), we can repeat all the steps in
\boxed{1} with $u$ replaced by $\uop$ and vice-versa to get
\begin{equation}\label{eq:quad_growth_2}
\cJ(\uop) \ge \cJ(u) + \pair{\cJ'(u),\uop-u} + \frac{\delta}{4} \norm{\uop-u}_{L^2(\Omega)}^2 .
\end{equation}
Adding \eqref{eq:quad_growth_1} and \eqref{eq:quad_growth_2} and setting $\epsilon = \frac{\delta}{8L}$ proves the corollary.
\end{proof}

\section{Discrete Control Problem}
\label{s:numerics}
Let $\cT$ denote a geometrically conforming, quasi-uniform triangulation of the domain $\Omega$ such that $\overline{\Omega} = \cup_{K\in\cT} K$ with $K$ closed and $h$ the meshsize of $\cT$. Consider the following finite dimensional spaces
\begin{align}\label{eq:disc_spaces}
\begin{aligned}
    Y^h &= \set{y_h \in C^0(\overline{\Omega}) : y_h\restriction_K \in \mathbb{P}_1(K), K \in \cT } , \\
    \mathring{Y}^h &= Y^h \cap \sobZ1\infty\Omega , \\
    U_{ad}^h &= Y^h \cap U_{ad} .
\end{aligned}
\end{align}
The spaces $U_{ad}^h$, $Y^h$ will be used to approximate the continuous solution of \eqref{eq:cost} and \eqref{eq:state}. The spaces are based on the finite dimensional space $\mathbb{P}_1$ which are the linear polynomials on the domain $K$, where $K$ is a triangle. This discretization is classical and can be found in any standard finite element book, for instance \cite{PGCiarlet_2002a,SCBrenner_RLScott_2008a}. We remark that in our numerical implementation the $L^p$ constraints in $U_{ad}^h$ are enforced by scaling the functions with their $L^p$-norm, we refer to \secref{s:computations} for more details. For the error analysis, we shall need the following.  Let $I_h : \sob1r\Omega \rightarrow Y^h$ be the global interpolation operator, i.e. if $r>n$ then $I_h$ is the standard Lagrange interpolation operator, otherwise it indicates the so-called Scott-Zhang interpolation operator \cite{RScott_SZhang_1990}.  Moreover, there exists a constant $C > 0$ independent of $h$ and $w$, such that $I_h$ satisfies the optimal estimate
\begin{align*}
    \normSZ{w - I_h w}1r\Omega \le C h \normSZ{w}2r\Omega , \quad \forall w \in \sob2r\Omega
        \quad 1 \le r \le \infty.
\end{align*}
We shall discretize the data $v$ using this Lagrange interpolant.

\subsection{Discrete State Equation}
The discrete state equation is given by
\begin{equation}
\label{eq:state_weak_discrete}
    \int_\Omega \frac{\grad y_h}{\cQ(y_h)} \cdot \grad z_h
        = \int_\Omega u_h z_h , \quad \mbox{ for all } z_h \in \mathring{Y}^h  .
\end{equation}
%
%
%
To prove the existence of a solution to the state equation \eqref{eq:state_weak_discrete}, as well as derive error estimates, we will borrow some ideas from \cite{RHNochetto_1989a}, which is motivated by \cite{DKindelehrer_GStampacchia_1980}. Let $R_1 = C_S  B_1$, where $C_S$ is the embedding constant of $\sob1\infty\Omega$ into $\sob2p\Omega$ and $B_1$ is taken from \eqref{eq:B}. We begin by modifying the the vector $G(y) = \frac{\grad y}{\cQ(y)}$ in the complement of $\set{x\in \mathbb{R}^2 : |x| > 3 R_1}$ as in \cite[p.~97]{DKindelehrer_GStampacchia_1980} and denote the new vector field by $\widetilde{G}$. The modification is such that the vector field $\widetilde{G}$ is strongly coercive. Let $y$ and $\widetilde{y}$ be the solutions to \eqref{eq:state} with $G$ and $\widetilde{G}$ respectively with right-hand-side $u_h$.  Essentially, $\widetilde{y}$ solves a regularized problem and provides a path to obtaining an error estimate between the solutions of \eqref{eq:state} and \eqref{eq:state_weak_discrete}.

To this end, we estimate the modulus of continuity $\omega$ of $\grad \widetilde{y}$ (and $\grad y)$.
\begin{lemma}[modulus of continuity]
\label{lem:modolus}
$\omega(\sigma) \le C \sigma^{1-2/p}$.
\end{lemma}
\begin{proof}
Using Morrey's inequality (see \cite[Lemma~4.1]{RHNochetto_1989a}), we know
\[
\norm{\grad \widetilde y}_{C^{0,\alpha}(\overline\Omega)} \le C(n,p,\theta,\Omega)  \norm{\widetilde y}_{W^2_p(\Omega)} ,
\quad \mbox{for } \alpha = 1 - 2/p .
\]
Then for $x_1,x_2 \in \Omega$, we get
\[
\abs{\grad \widetilde y(x_1) - \grad \widetilde y(x_2)} \le C \abs{x_1-x_2}^{1-2/p} ,
\]
which implies the assertion.
\end{proof}

The following lemma provides an estimate of the $L^\infty$ norm of $\grad(y - \widetilde y)$.
\begin{lemma}
\label{lem:loc_estimate}
If $\eta > 0$ and $\norm{y - \widetilde{y}}_{L^\infty(\Omega)} < \eta$, then
\[
\norm{\grad(y-\widetilde{y})}_{L^\infty(\Omega)} \le
\left\{\begin{array}{ll}
        C \eta^{1-2/p} & p < 4 , \\
        C \eta^{1/2}     & p \ge 4 .
       \end{array}
\right.
\]
Moreover, $\widetilde{y}$ solves \eqref{eq:state} with $\norm{\grad \widetilde y}_{L^\infty(\Omega)}\le 2 R_1$.
\end{lemma}
\begin{proof}
With $\norm{y - \widetilde{y}}_{L^\infty(\Omega)} < \eta$ and Lemma~\ref{lem:modolus} in hand, the proof is based on \cite[Lemma~4.2]{RHNochetto_1989a} (we provide the details for completeness). Let $x_0 \in \Omega$ be such that $\norm{\grad e}_{L^\infty(\Omega)} = \abs{\grad e(x_0)}$, where $e = y - \widetilde{y}$. Then, for $x \in \Omega \cap \set{x \in \mathbb{R}^2 : \abs{x - x_0} = \eta^{1/2}}$, the Fundamental theorem of calculus gives
\begin{align*}
e(x)
&= e(x_0) + \int_0^1 \grad e(s x + (1-s)x_0) ds \cdot (x-x_0) \\
&= e(x_0) + \grad e(x_0) (x-x_0) + \int_0^1 \del{\grad e(s x + (1-s)x_0) - \grad e(x_0)} ds \cdot (x-x_0) .
\end{align*}
This leads to
\[
\eta^{1/2} \abs{\grad e(x_0)} \le \abs{e(x)} + \abs{e(x_0)}
    + \omega(\abs{x-x_0}) \cdot \abs{x-x_0} .
\]
Using \lemref{lem:modolus} we deduce
\[
 \abs{\grad e(x_0)} \le C \del{\eta^{1/2} + \eta^{1-2/p}} .
\]
For $p < 4$, we get $\norm{\grad e}_{L^\infty(\Omega)} \le C \eta^{1-2/p}$, otherwise $\norm{\grad e}_{L^\infty(\Omega)} \le C \eta^{1/2}$.
For sufficiently small $\eta > 0$, we have $\abs{\grad \widetilde y(x_0)} \le 2 R_1$ for all
$x_0 \in \Omega$. As $\widetilde{G}(\grad \widetilde y(x)) = G(\grad \widetilde y(x))$ for all $x \in \Omega$,
therefore $\widetilde y$ solves \eqref{eq:state}.
\end{proof}
We thus have the following result.
\begin{theorem}[existence of the discrete solution]
\label{thm:error_state}
Let $y, \widetilde y_{h}$ solve \eqref{eq:state}, \eqref{eq:state_weak_discrete} with $G$ and $\widetilde{G}$ respectively and right-hand-side $u_h$. There exists $h_0 > 0$ such that for all $0 < h \le h_0$, the problem \eqref{eq:state_weak_discrete} admits
a unique solution $\widetilde y_{h}$. Setting $\eta = h^{p/(p-2)}$ for $p < 4$, and $\eta = h^2$ for $p \geq 4$, we get
\[
\norm{\grad(y-\widetilde y_{h})}_{L^\infty(\Omega)} \le C(n,p,R_1,\theta,\Omega) h \abs{\log h}^4.
\]
\end{theorem}
\begin{proof}
We proceed in two steps:

\noindent
\boxed{1}
Let $\widetilde y_{h} \in I_h v \oplus \mathring{Y}^h$ be the solution to
\eqref{eq:state_weak_discrete} with $\widetilde{G}$ instead of $G$. Then using \cite[Theorem~3.2]{RHNochetto_1989a}, we obtain
\begin{equation*}
\norm{\grad(\widetilde y - \widetilde y_{h})}_{L^\infty(\Omega)} \le C h \abs{\log h}^4 ,
\end{equation*}
which, using \lemref{lem:loc_estimate} and $h>0$ sufficiently small, immediately implies $\norm{\grad \widetilde y_{h}}_{L^\infty(\Omega)} \le 3 R_1$; thus, $\widetilde y_{h}$ is the solution to the discrete problem with $G$ instead of $\widetilde{G}$.

\noindent
\boxed{2}
Using the triangle inequality, we get
$
\norm{\grad(y-\widetilde y_{h})}_{L^\infty(\Omega)}
\le \norm{\grad(y-\widetilde y)}_{L^\infty(\Omega)}
+ \norm{\grad(\widetilde y-\widetilde y_{h})}_{L^\infty(\Omega)}
$.
Then Lemma~\ref{lem:loc_estimate} and \boxed{1} gives the estimate.
\end{proof}

\subsection{Discrete Optimal Control Problem}

We first recall that $\uop$ denotes the local optimal control for the continuous problem \eqref{eq:cost}. The discrete version of the continuous optimal control problem \eqref{eq:cost} is
\begin{equation} \label{eq:cost_discrete}
    \inf \cJ_h\del{y_h,u_h} :=
    \frac{1}{2}\norm{y_h-y_d}_{L^2(\Omega)}^2
      + \frac{\alpha}{2} \norm{u_h}_{L^2(\Omega)}^2  \quad \mbox{over }
        y_h - I_h v \in \mathring{Y}^h , u_h \in U_{ad}^h ,
\end{equation}
subject to $y_h - I_h v \in \mathring{Y}_h$ solving \eqref{eq:state_weak_discrete}.
We remark that in \eqref{eq:cost_discrete}, for simplicity, we have not discretized $y_d$.

The discrete optimality conditions amount to the state \eqref{eq:state_weak_discrete}; the adjoint, find $\adjop_h \in \mathring{Y}^h$ such that
\begin{equation}\label{eq:adjoint_weak_discrete}
    \int_\Omega \grad z^T_h A[\yop_h] \grad \adjop_h = \int_\Omega (\yop_h-y_d)z_h \quad \mbox{ for all } z_h \in \mathring{Y}^h ,
\end{equation}
where $A[\yop_h] = \frac{1}{\cQ(\yop_h)}\del{ \cI - \frac{\grad \yop_h \grad \yop_h^T }{\cQ(\yop_h)^2} }$,
and the discrete variational inequality for the optimal control
\begin{equation}\label{eq:control_weak_discrete}
    \pair{\adjop_h + \alpha \uop_h, u_h-\uop_h}_{L^2(\Omega),L^2(\Omega)} \ge 0, \quad
        \mbox{ for all } u_h \in U_{ad}^h .
\end{equation}

\begin{remark}\label{rem:functional_derivative_discrete}
Similar to Remark \ref{rem:functional_derivative}, the discrete functional derivative is given by $\cJ_h'(u_h) = \adjoint_h(y_h) + \alpha u_h$ for an arbitrary $u_h$ in $U_{ad}^h$, where $y_h$ solves \eqref{eq:state_weak_discrete} with $u_h$ as right-hand-side, and $\adjoint_h(y_h)$ solves \eqref{eq:adjoint_weak_discrete} with right-hand-side given by $y_h - y_d$.
\end{remark}

The notion of \emph{local} control is useful for making sense of the error estimate on the optimal control.
\begin{definition}[local control]
A control $\uop_h \in U_{ad}^h$ is a \emph{local} solution to \eqref{eq:cost_discrete}, if
\[
\cJ_h(u_h) \ge \cJ_h(\uop_h) \quad \mbox{for all } u_h \in U_{ad}^h \mbox{ with }
\norm{u_h - \uop}_{L^p(\Omega)} \le \gamma
\]
holds for a certain $\gamma > 0$.  Note that $\uop$ appears.
\end{definition}

To this end, we make the following assumption.
\begin{assumption}\label{ass:local_exist_disc_contr}
There exists $\uop_h \in U^h_{ad}$ which is a local solution to \eqref{eq:cost_discrete}.
\end{assumption}
\begin{remark}
If instead of $\mathbb{P}_1$, we use $\mathbb{P}_0$ to define $U_{ad}^h$ in \eqref{eq:disc_spaces}, Assumption~\ref{ass:local_exist_disc_contr} can be shown based on \cite[Section~4.4]{ECasas_FTroeltzsch_2002a}. This uses the fact that the standard $L^2$-orthogonal projection operator
$Q_h : L^2(\Omega) \rightarrow \mathbb{P}_0$, defined by $(u - Q_h u, u_h) = 0$, for all $u_h \in U_{ad}^h$, is
given by $Q_h u \restriction_{K} = \del{\frac{1}{\abs{K}} \int_K u}  \in U_{ad}^h$. The final inclusion is due to the well-known Jensen's inequality.
\end{remark}

We next state an important intermediate estimate for the optimal control.
\begin{theorem}[error estimate on the control]\label{thm:prelim_estimate_control}
    Let $\varphi(\uop_h)$ solve the continuous adjoint equation \eqref{eq:adjoint_strong}, with continuous state corresponding to $\uop_h$, and $\adjop_h(\uop_h)$ solve the discrete adjoint equation \eqref{eq:adjoint_weak_discrete}, with discrete state corresponding to $\uop_h$.
    Under Assumptions~\ref{ass:ssc} and \ref{ass:local_exist_disc_contr}, there exists a constant $C > 0$ such that
    \begin{equation}\label{eq:control_estimate}
       \delta \normLt{\overline{u} - \overline{u}_h}\Omega \le C \norm{\varphi(\uop_h)-\adjop_h(\uop_h)}_{L^2(\Omega)} .
    \end{equation}
\end{theorem}
\begin{proof}
    The proof is based on \cite{HAntil_RHNochetto_PSodre_2013a}, we only state the key steps here. The idea is to
    replace $u$ by $\overline{u}_h$ in \eqref{eq:control_weak} and $u_h$ by $P_h \overline{u}$ in \eqref{eq:control_weak_discrete}, where $P_h$ is the $L^2$ orthogonal projection onto $U_{ad}^h$.  This gives
    \begin{equation}\label{eq:control_both}
        \pair{\cJ'(\uop), \uop_h - \uop} \ge 0 , \quad
        \pair{\cJ_h'(\uop_h), P_h\uop - \uop_h} \ge 0 .
    \end{equation}
    Using \eqref{eq:gradJ_quad_growth}, and replacing $u$ by $\uop_h$ (here we use Assumption~\ref{ass:local_exist_disc_contr}), we have
    \[
        \frac{\delta}{2} \normLt{\uop_h-\uop}{\Omega}^2
			 \leq \pairLt{\cJ'(\uop_h)  - \cJ'(\uop), \uop_h - \uop}{\Omega}  .
    \]
    Adding and subtracting $\cJ'_h(\uop_h)$ followed by using first inequality in \eqref{eq:control_both} we obtain
    \begin{align*}	
			\frac{\delta}{2} \normLt{\uop_h-\uop}{\Omega}^2 &\le \pairLt{\cJ'(\uop_h) - \cJ_h'(\uop_h), \uop_h - \uop}{\Omega}
			+ \pairLt{\cJ_h'(\uop_h), \uop_h - \uop}{\Omega} .
       \end{align*}
    Adding and subtracting $P_h \uop$ to $\uop_h-u$ in the second term, and using the fact that $P_h$ is an orthogonal projection, we have $\pair{\cJ'_h(\uop_h),P_h \uop - \uop} = 0$. Therefore, invoking the second inequality in \eqref{eq:control_both}, we deduce \eqref{eq:control_estimate} from Remark \ref{rem:functional_derivative_discrete} and the Cauchy-Schwarz inequality.
\end{proof}
%
%

It is clear from \thmref{thm:prelim_estimate_control} that in order to prove the estimate for the control we need to estimate the solution to the continuous and discrete adjoint equations but both for the discrete optimal control $\uop_h$. In view of \eqref{eq:adjoint_strong} and \eqref{eq:adjoint_weak_discrete}, we need to estimate the solution to the continuous state equation $\yop(\uop_h)$ and the discrete state equation $\yop_h(\uop_h)$ both for the discrete control $\uop_h$. In the sequel, we use such an estimate from \thmref{thm:error_state}, but first we derive an estimate for the adjoint.

\begin{lemma}[error estimate on the adjoint]
\label{lem:error_estimate_adjoint}
Let $\varphi$ solve \eqref{eq:adjoint_strong}, with right-hand-side $y - y_d$, and $\varphi_h$ solve \eqref{eq:adjoint_weak_discrete}, with right-hand-side $y_h - y_d$.  Then there exists a constant $C > 0$ such that
\begin{equation}\label{eq:estimate_adjoint}
    \normSZ{\adjoint - \adjoint_h}12\Omega \le C \del{h \normS{\adjoint}22\Omega + \normSZ{y-y_h}1\infty\Omega} .
\end{equation}
\end{lemma}
\begin{proof}
Using the discrete inf-sup condition from \cite[Proposition~8.6.2]{SCBrenner_RLScott_2008a} and $I_h$, we have
\[
    \normSZ{\adjoint_h - I_h \adjoint}12\Omega
        \le C \sup_{z \in \mathring{Y}^h} \frac{\int_\Omega \grad z^T A(y_h) \grad (\adjoint_h -I_h \adjoint)}{\normSZ{z}12\Omega} .
\]
In view of \eqref{eq:adjoint_weak_discrete} we obtain
\begin{align*}
\normSZ{\adjoint_h - I_h \adjoint}12\Omega
    &\le C \sup_{z \in \mathring{Y}^h} \frac{\int_\Omega (y_h - y_d) z -  \grad z^T A(y_h) \grad I_h \adjoint}{\normSZ{z}12\Omega} \\
    &= C \sup_{z \in \mathring{Y}^h} \frac{\int_\Omega (y_h - y_d) z - (y-y_d)z + \grad z^T A(y) \grad \adjoint -  \grad z^T A(y_h) \grad I_h \adjoint}{\normSZ{z}12\Omega}
\end{align*}
where the last equality follows immediately using \eqref{eq:adjoint_strong}. Invoking Cauchy-Schwarz, we readily obtain
\begin{align*}
    \normSZ{\adjoint_h - I_h \adjoint}12\Omega
        &\le C \Big( \normS{y_h-y}{-1}2\Omega
         + \norm{A(y)}_{L^\infty(\Omega)} \normSZ{\adjoint - I_h \adjoint}12\Omega \\
        &\quad + \norm{A(y)-A(y_h)}_{L^\infty(\Omega)} \normSZ{I_h \adjoint}12\Omega \Big) .
\end{align*}
In view of \lemref{lem:A_Lipschitz} we deduce
\[
    \normSZ{\adjoint_h - I_h \adjoint}12\Omega \le C \del{h \normS{\adjoint}22\Omega + \normSZ{y-y_h}1\infty\Omega} .
\]
The estimate \eqref{eq:estimate_adjoint} follows readily using triangle inequality.
\end{proof}
\begin{corollary}\label{cor:final_error_est}
Let Assumptions~\ref{ass:ssc} and \ref{ass:local_exist_disc_contr} hold. Furthermore, let $\adjop(\uop_h)$ be the solution of the continuous adjoint equation \eqref{eq:adjoint_strong} and $\yop(\uop_h)$ the solution of the continuous state equation \eqref{eq:state_weak} with control $\uop_h$. Furthermore, let $\adjop_h(\uop_h)$ be the solution of the discrete adjoint equation \eqref{eq:adjoint_weak_discrete} and $\yop_h(\uop_h)$ the solution of the discrete state equation \eqref{eq:state_weak_discrete} with control $\uop_h$. If $h \le h_0$, for $h_0 > 0$ sufficiently small, then there is a constant $C \ge 1$ depending on $\normS{\yop}2p\Omega$, $\normS{\adjop}22\Omega$, $\norm{y_d}_{L^p(\Omega)}$, such that
\[
    \normSZ{y(\uop_h)-\yop_h(\uop_h)}1\infty\Omega + \normSZ{\varphi(\uop_h)-\adjop_h(\uop_h)}12\Omega + \delta \norm{\uop-\uop_h}_{L^2(\Omega)} \le C h |\log h|^4 .
\]
\end{corollary}
\begin{proof}
By Poincar\'{e}, we have
$
    \norm{\varphi(\uop_h)-\adjop_h(\uop_h)}_{L^2(\Omega)}
     \le C_0 \normSZ{\varphi(\uop_h)-\adjop_h(\uop_h)}12\Omega
$.
Then, combining \lemref{lem:error_estimate_adjoint} with \thmref{thm:error_state}, we deduce
\[
    \normSZ{\varphi(\uop_h) - \adjop_h(\uop_h)}12\Omega \le C^* h \abs{\log h}^4,
\]
with constant $C^*$ having the same dependencies as $C$. This, together with \eqref{eq:control_estimate}, implies the estimate for the control $\norm{\uop-\uop_h}_{L^2(\Omega)}$.  The remaining estimates follow immediately.
\end{proof}
%

\section{Numerical Examples}
\label{s:computations}

\subsection{Setup}

We present numerical examples for the discrete optimal control problem in Section \ref{s:numerics}.  We solve the optimization problem using MATLAB's optimization toolbox with an SQP method, where we provide the gradient information.

The gradient of the cost functional \eqref{eq:cost_discrete}, at each iteration of the optimization algorithm, is computed by first solving the state equation \eqref{eq:state_weak_discrete} for $y_h$ with the control $u_h$ taken from the previous iteration.  Then, the adjoint problem \eqref{eq:adjoint_weak_discrete} is solved for $\adjoint_h$ using the discrete solution $y_h$.  We then define the linear form (see Remark \ref{rem:functional_derivative_discrete})
\begin{equation*}
  \pair{\cJ_h'(u_h), v_h}_{L^2(\Omega),L^2(\Omega)} = \int_{\Omega} (\adjoint_h + \alpha u_h) v_h, \quad \text{for all } v_h \in Y^h,
\end{equation*}
and pass the discrete gradient vector (and cost value) to MATLAB's optimization algorithm at the current iteration.  The constraint on the control $U_{ad}^h$ is handled by MATLAB's optimization algorithm by specifying an inequality constraint on $u_h$.

The non-linear state equation is solved with Newton's method and a direct solver (backslash); we also use a direct solver for the adjoint problem.  This was all implemented in MATLAB using the FELICITY toolbox \cite{FELICITY_REF}.  The following sections show some examples of our computational method.  In all cases, we set $\alpha = 10^{-6}$ and $p = 2.5$.  For most examples, we set $\theta = 20$ in the definition of $U_{ad}^h$, except in Section \ref{sec:sine_on_square_theta=2} where $\theta=2$.  The first two examples are posed on a unit square domain, which technically does not satisfy the $C^{1,1}$ domain assumption.  The last example is posed on a $C^{\infty}$ domain in the shape of a four-leaf clover.

\subsection{Sine On A Square}

\subsubsection{$\theta = 20$}\label{sec:sine_on_square_theta=20}

We take $y_d$ to be a product of sine functions and set the boundary data to $v = 0$.  The domain $\Omega$ is the unit square.  See Figures \ref{fig:Sine_On_Square_Surface} and \ref{fig:Sine_On_Square_Optim} for plots of $y_d$, $\yop$, $\uop$, and the optimization history.  This example shows that we can recover the desired surface almost exactly when the boundary condition $v$ matches $y_d$ on $\partial \Omega$.  Note: for this optimal control, we have $\norm{\uop}_{L^p(\Omega)} \approx 3.75$.
\begin{figure}
\begin{center}
\subfloat{


\includegraphics[width=3.3in]{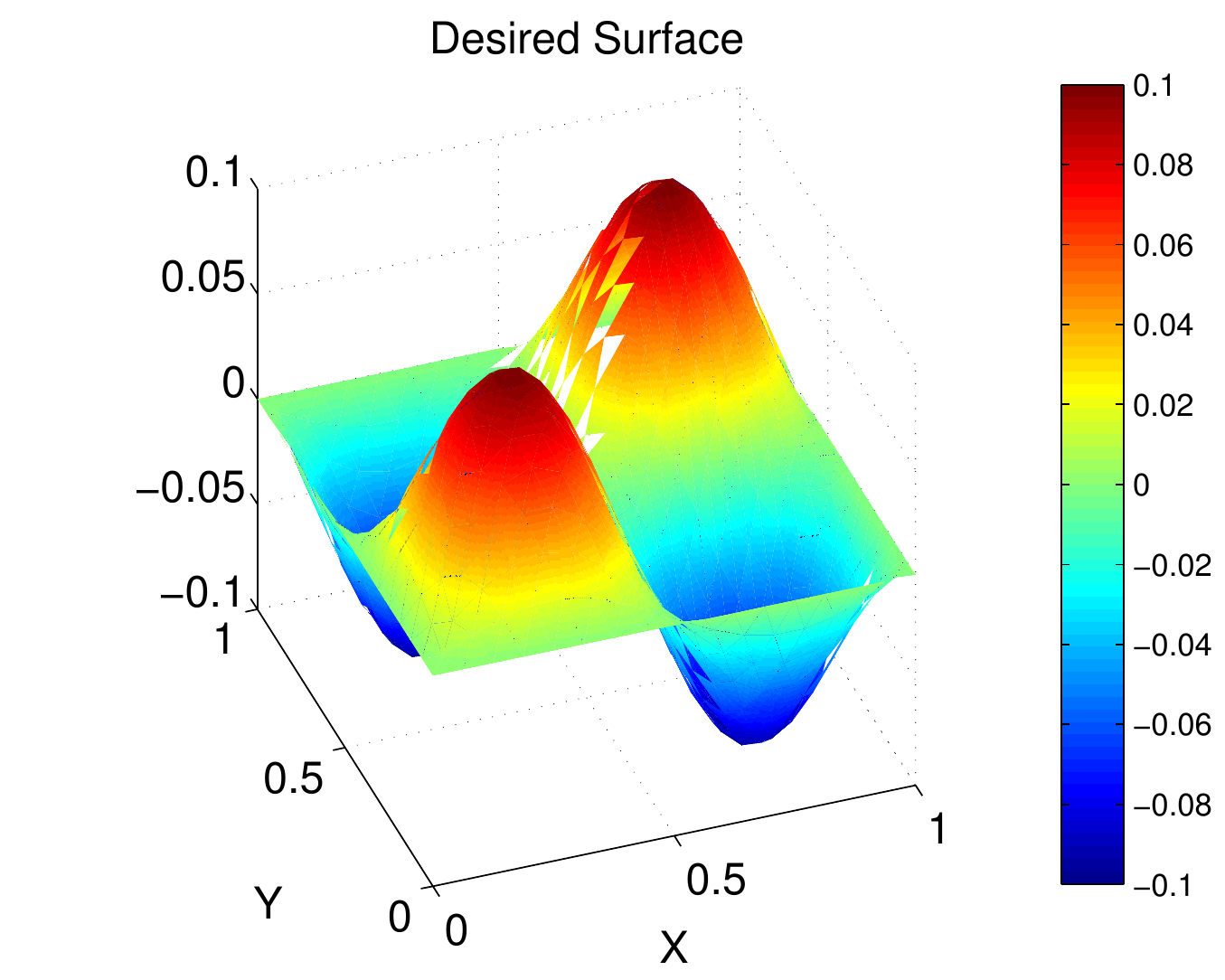}} 
\hspace{-0.3in}
\subfloat{


\includegraphics[width=3.3in]{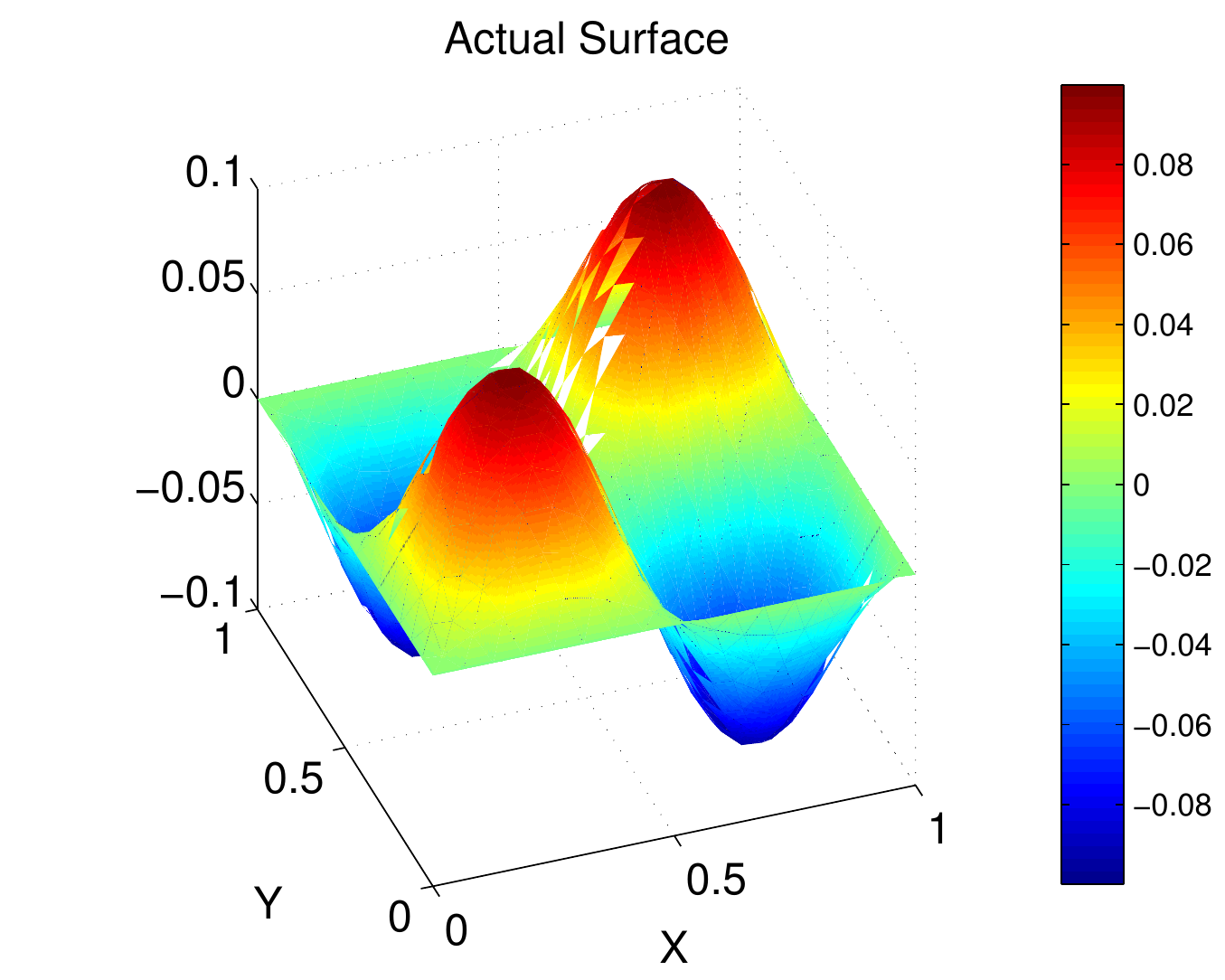}} 
\caption{Left: Desired surface height $y_d = \sin(2 \pi x) \sin(2 \pi y)$.  Right: Actual surface height $\yop$ (after the optimization method converges).  Boundary data is $v = 0$.}
\label{fig:Sine_On_Square_Surface}
\end{center}
\end{figure}
\begin{figure}
\begin{center}
\subfloat{


\includegraphics[width=3.2in]{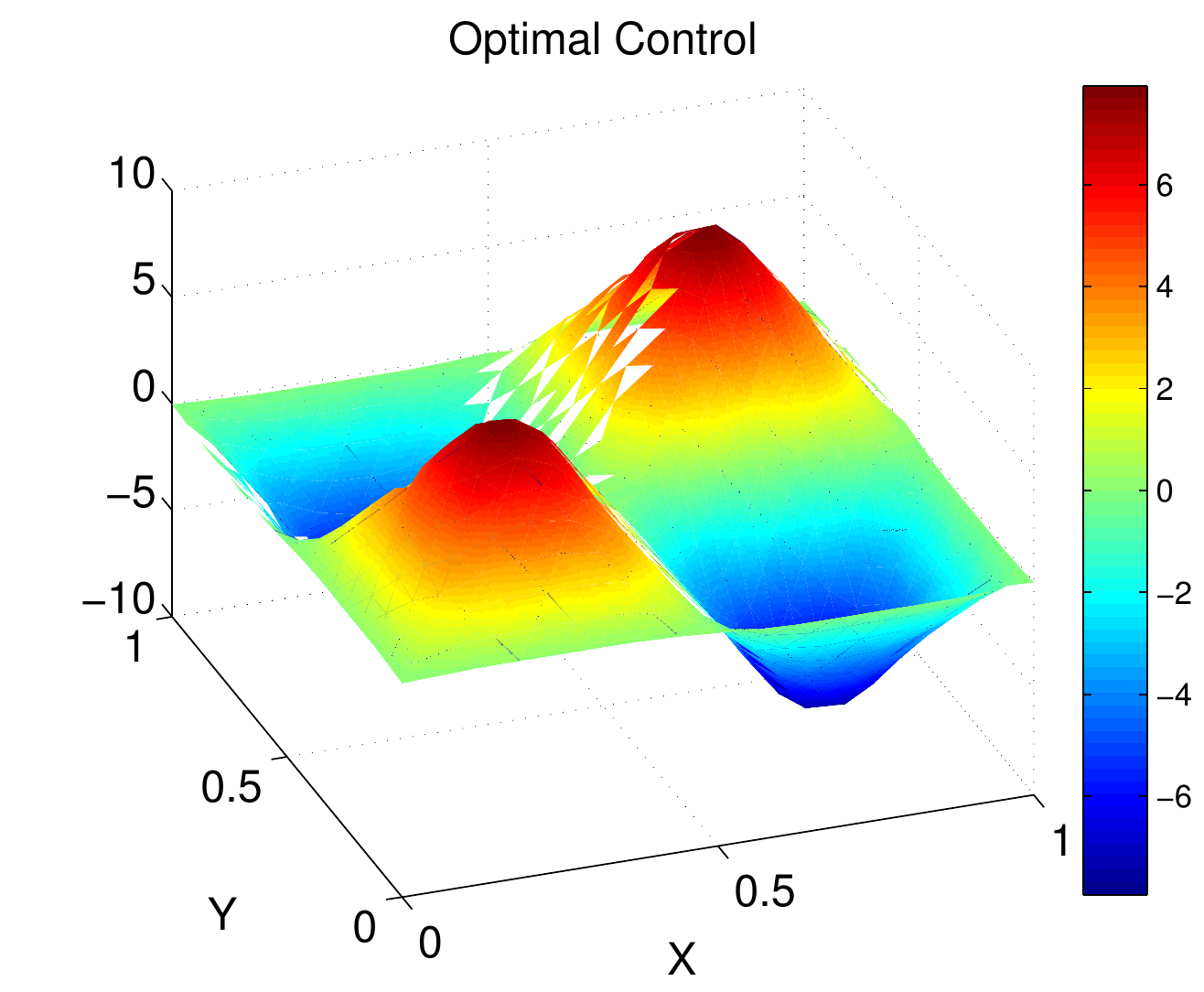}} 
\subfloat{


\includegraphics[width=3.2in]{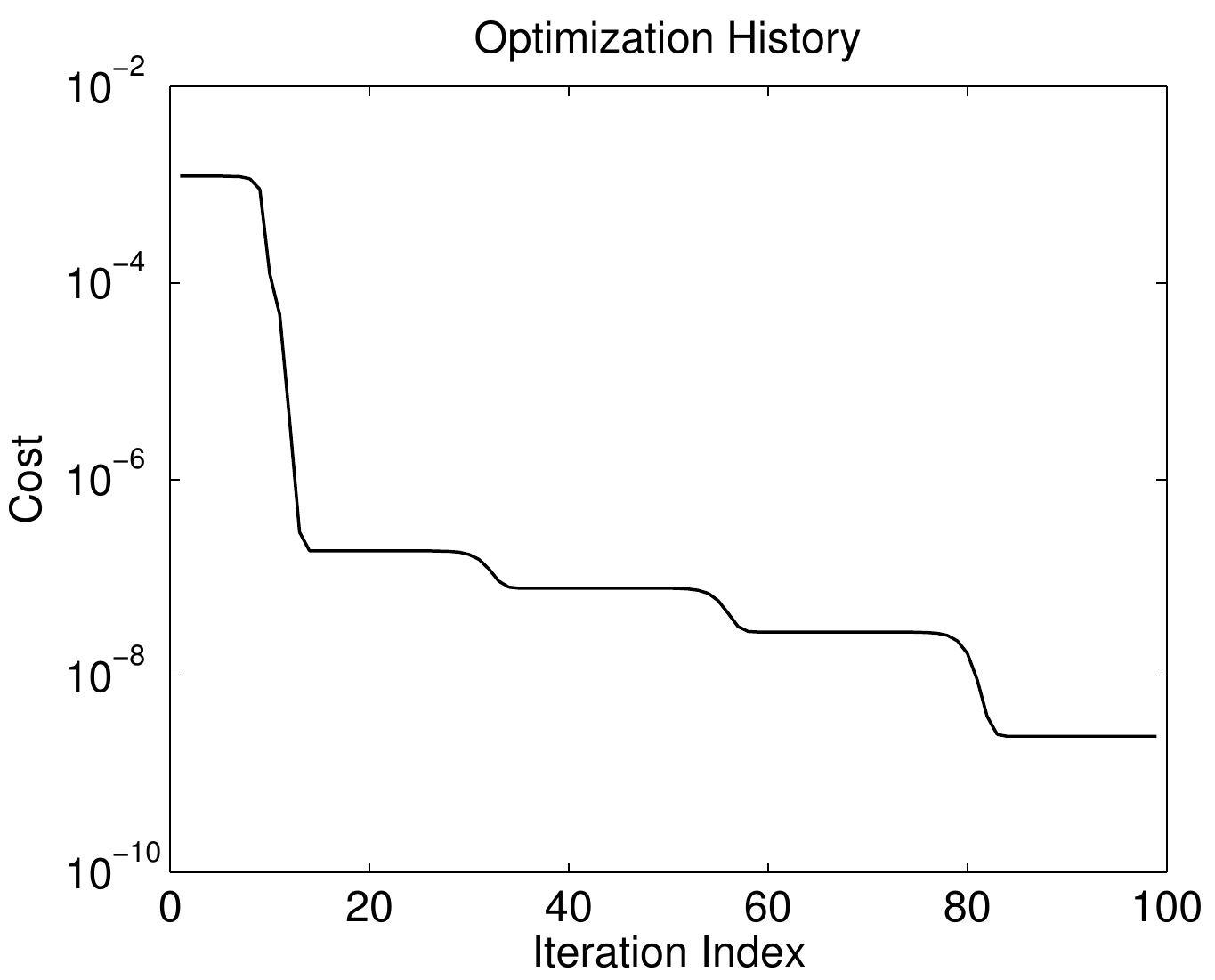}} 
\caption{Left: Optimal control function $\uop$ for $y_d$ in Figure \ref{fig:Sine_On_Square_Surface}.  Right: Decrease of cost functional $\cJ$.}
\label{fig:Sine_On_Square_Optim}
\end{center}
\end{figure}

\subsubsection{$\theta = 2$}\label{sec:sine_on_square_theta=2}

\begin{figure}
\begin{center}
\subfloat{


\includegraphics[width=3.3in]{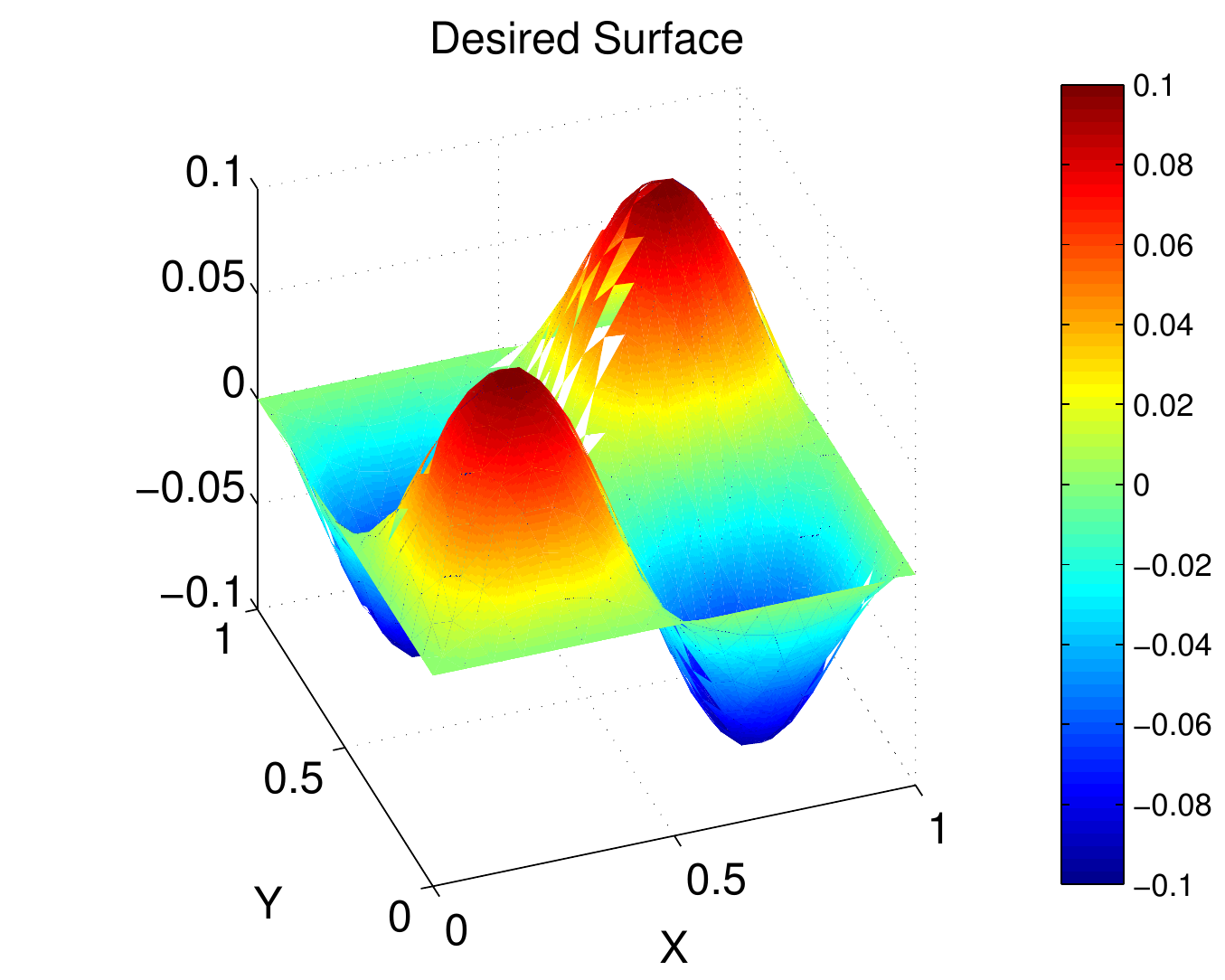}}
\hspace{-0.3in}
\subfloat{


\includegraphics[width=3.3in]{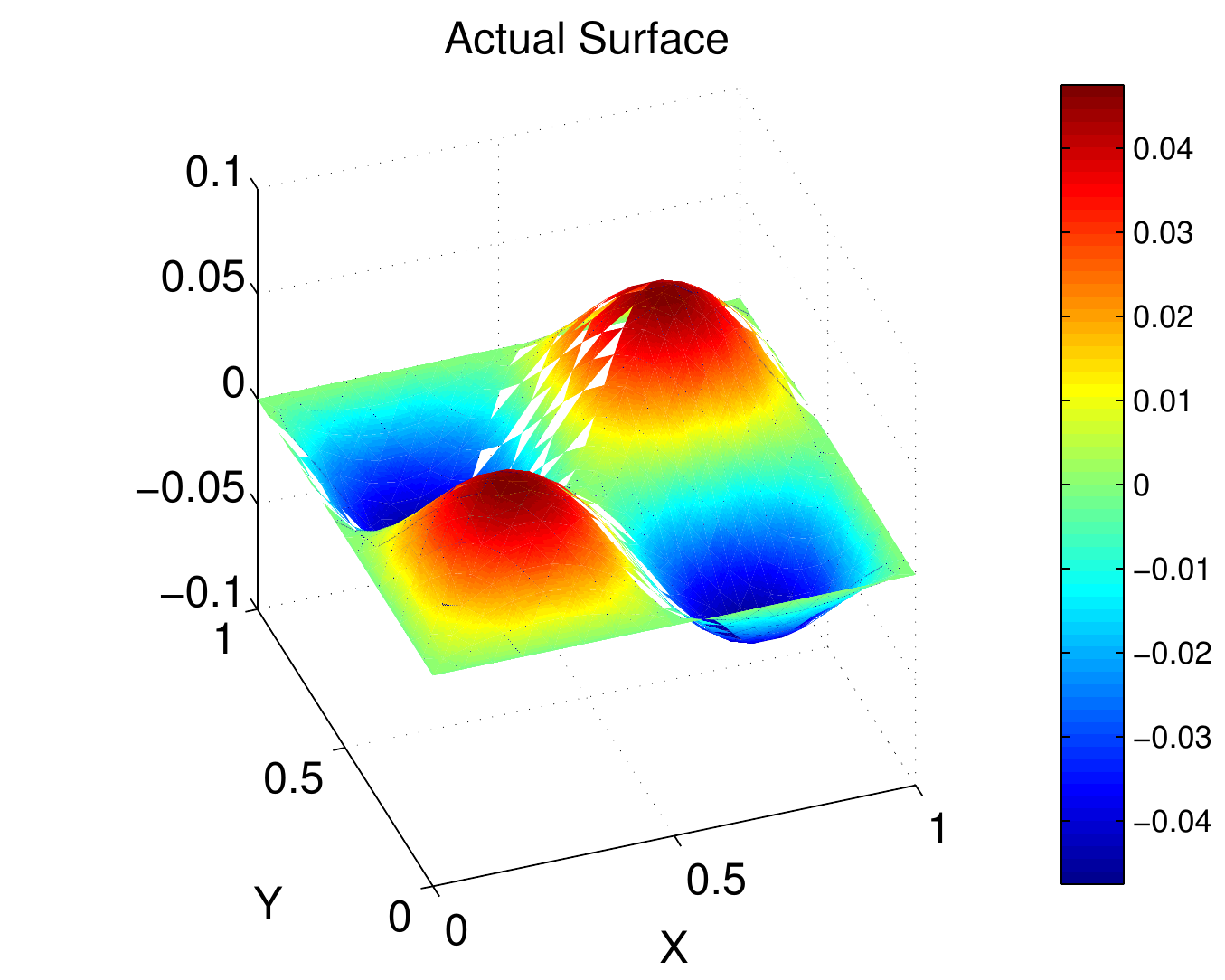}}
\caption{Left: Desired surface height $y_d = \sin(2 \pi x) \sin(2 \pi y)$.  Right: Actual surface height $\yop$ (after the optimization method converges).  Boundary data is $v = 0$.  Note: $\theta=2$ here.}
\label{fig:Sine_On_Square_Small_Theta_Surface}
\end{center}
\end{figure}
\begin{figure}
\begin{center}
\subfloat{


\includegraphics[width=3.2in]{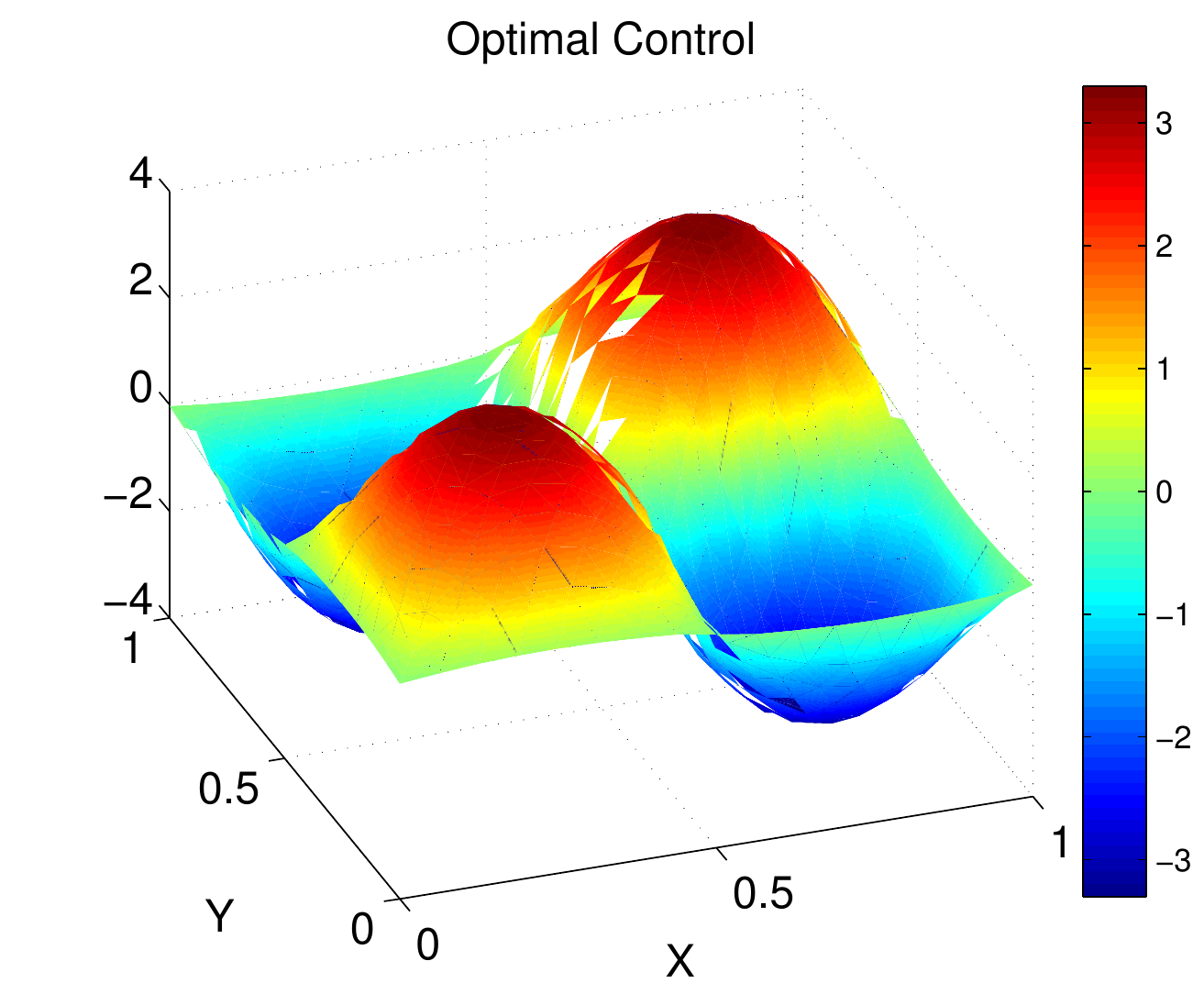}}
\subfloat{


\includegraphics[width=3.2in]{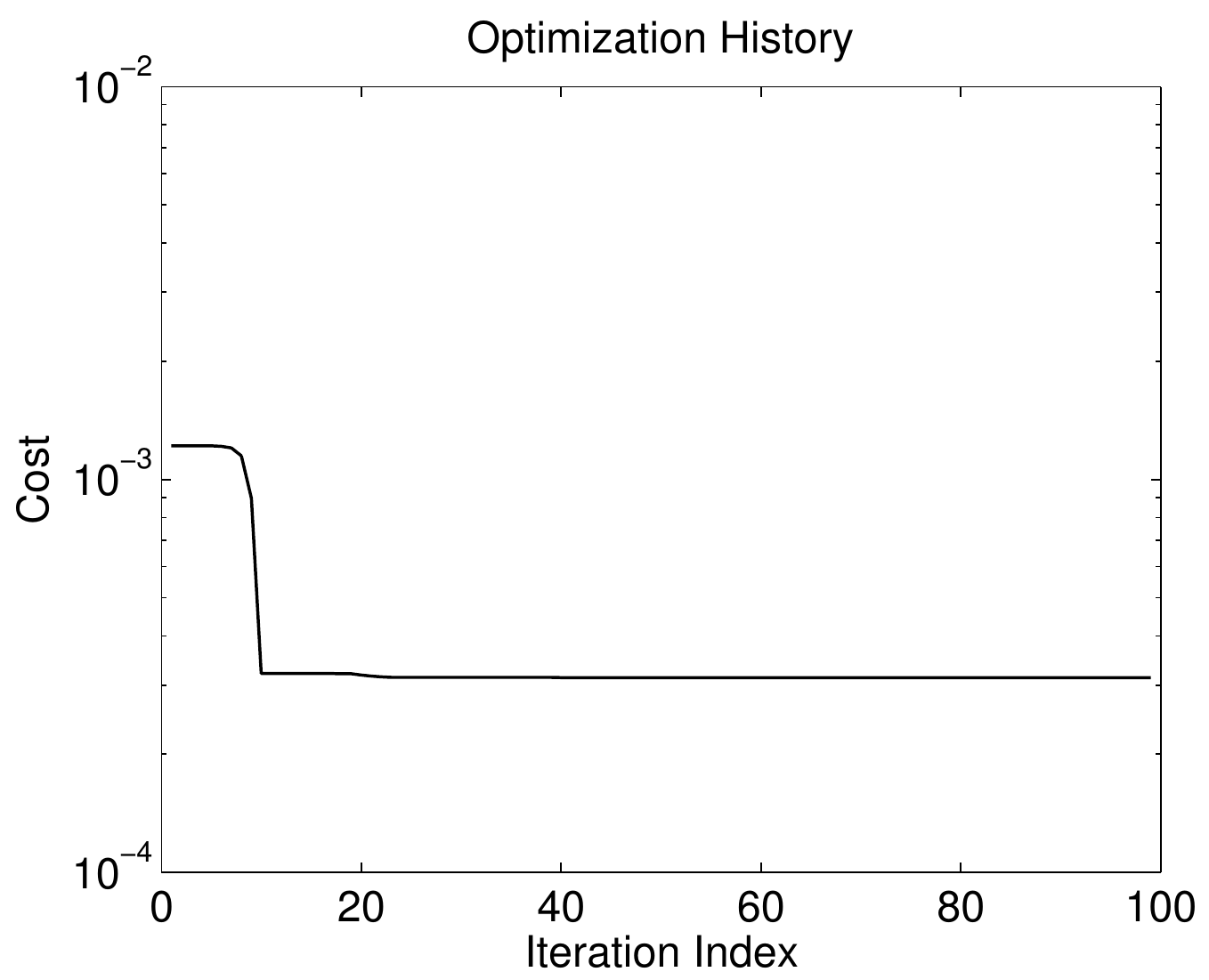}}
\caption{Left: Optimal control function $\uop$ for $y_d$ in Figure \ref{fig:Sine_On_Square_Small_Theta_Surface}.  Right: Decrease of cost functional $\cJ$.}
\label{fig:Sine_On_Square_Small_Theta_Optim}
\end{center}
\end{figure}
We run the same example as in Section \ref{sec:sine_on_square_theta=20}, except we choose a smaller value of $\theta$ to see the impact on the quality of the optimal control; all other parameters are identical.  See Figures \ref{fig:Sine_On_Square_Small_Theta_Surface} and \ref{fig:Sine_On_Square_Small_Theta_Optim} for plots of $y_d$, $\yop$, $\uop$, and the optimization history.  The value of $\norm{\uop}_{L^p(\Omega)}$ in the previous example was $\approx 3.75$.  Here, $\norm{\uop}_{L^p(\Omega)}$ is constrained to be $\leq 2$ (in fact, it is equal to $2$).

It is clear from Figure \ref{fig:Sine_On_Square_Small_Theta_Optim} that the height of the optimal control is less than in Figure \ref{fig:Sine_On_Square_Optim} (note the different scale in the plot).  Moreover, $\uop$ is not as ``peaked'' as before (more rounded), but is qualitatively the same.  This, in turn, affects the obtained surface height $\yop$ in Figure \ref{fig:Sine_On_Square_Small_Theta_Surface}, i.e. it appears to be uniformly scaled with respect to the result in Figure \ref{fig:Sine_On_Square_Surface}.  In other words, the main effect that $\theta$ has is to \emph{scale down} the optimal control, which shrinks the obtained surface height.  But the qualitative shape of $\uop$ and $\yop$ is essentially the same as before.

\subsection{Gaussian On A Square (Nonzero Boundary Condition)}

We take $y_d$ to be a Gaussian bump and set the boundary data to $v = -0.1 \sin (\pi x) \cos (2\pi y)$.  The domain $\Omega$ is the unit square. See Figures \ref{fig:Gaussian_On_Square_Nonzero_BC_Surface} and \ref{fig:Gaussian_On_Square_Nonzero_BC_Optim} for plots of $y_d$, $\yop$, $\uop$, and the optimization history.  In this case, we impose a mismatch between the imposed boundary condition $v$ and the desired surface $y_d$.  The results show that the optimization does the ``best it can'' by trying to match $y_d$ in the interior of $\Omega$.  Note the large value of the control $\uop$ at the boundary of $\Omega$ in Figure \ref{fig:Gaussian_On_Square_Nonzero_BC_Optim}.
\begin{figure}
\begin{center}
\subfloat{


\includegraphics[width=3.3in]{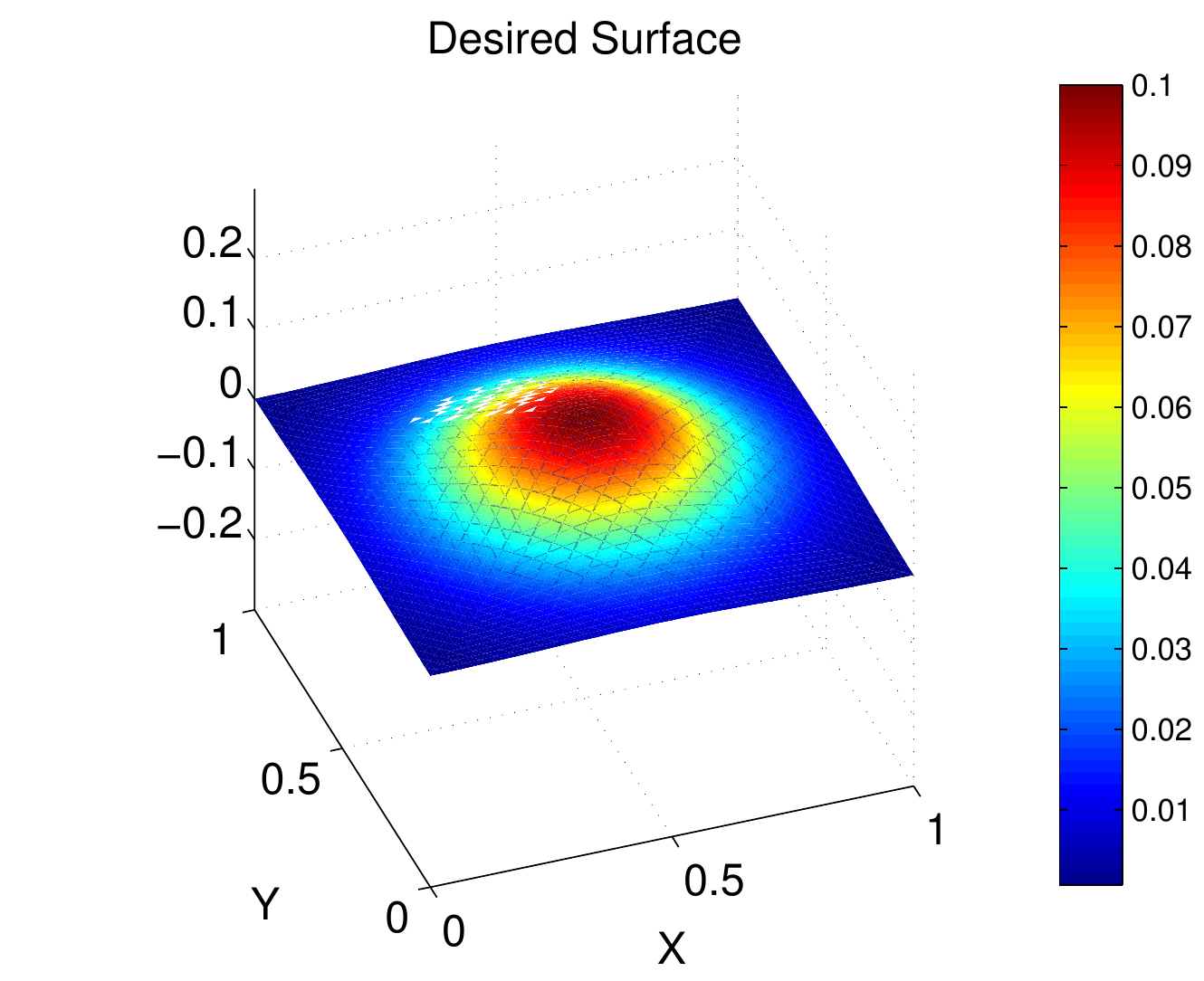}}
\hspace{-0.3in}
\subfloat{


\includegraphics[width=3.3in]{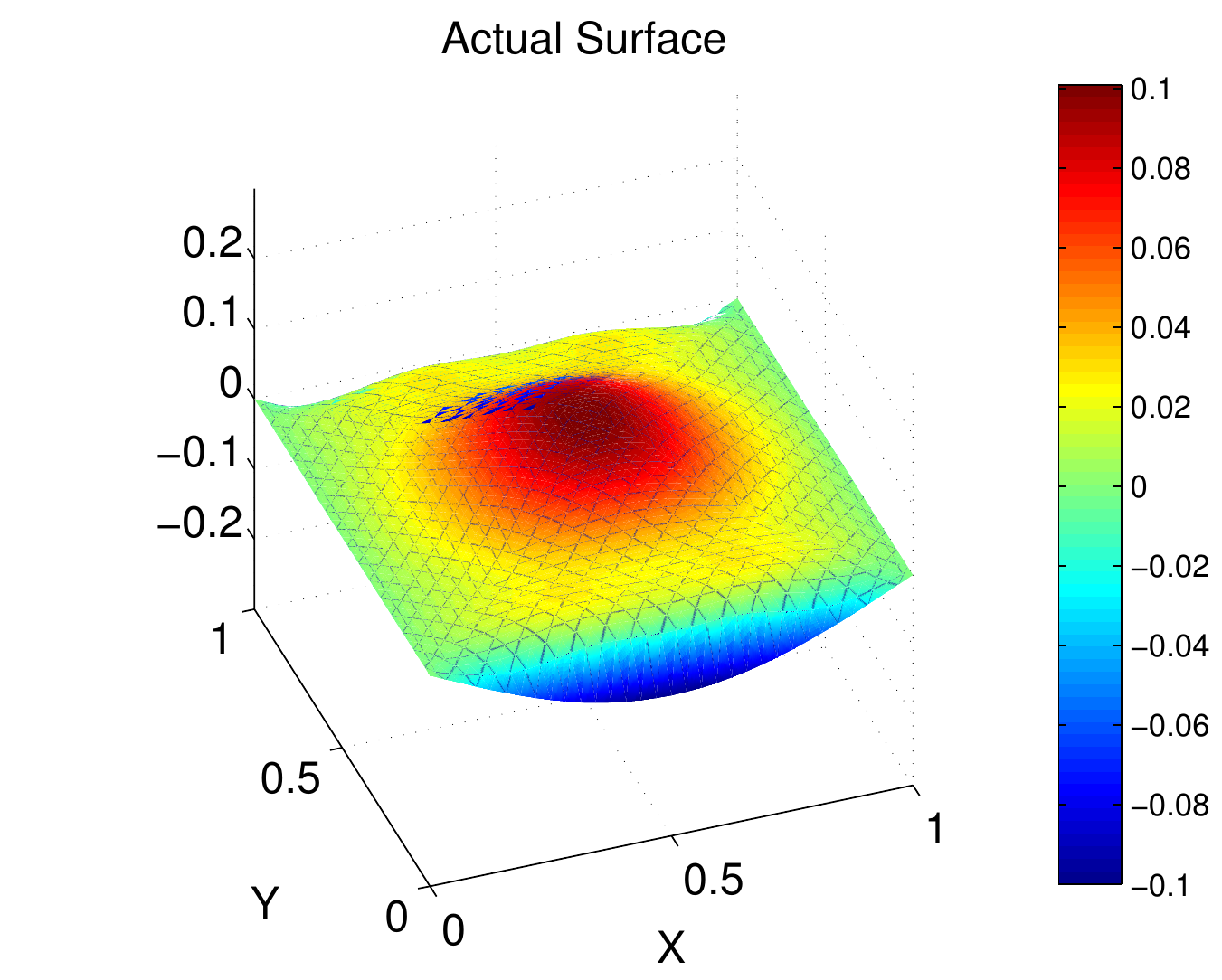}}
\caption{Left: Desired surface height $y_d = 0.1 \exp \left( -((x-0.5)^2 + (y-0.5)^2)/0.1 \right)$.  Right: Actual surface height $\yop$ (after the optimization method converges).  Boundary data is $v = -0.1 \sin (\pi x) \cos (2\pi y)$.}
\label{fig:Gaussian_On_Square_Nonzero_BC_Surface}
\end{center}
\end{figure}
\begin{figure}
\begin{center}
\subfloat{


\includegraphics[width=3.2in]{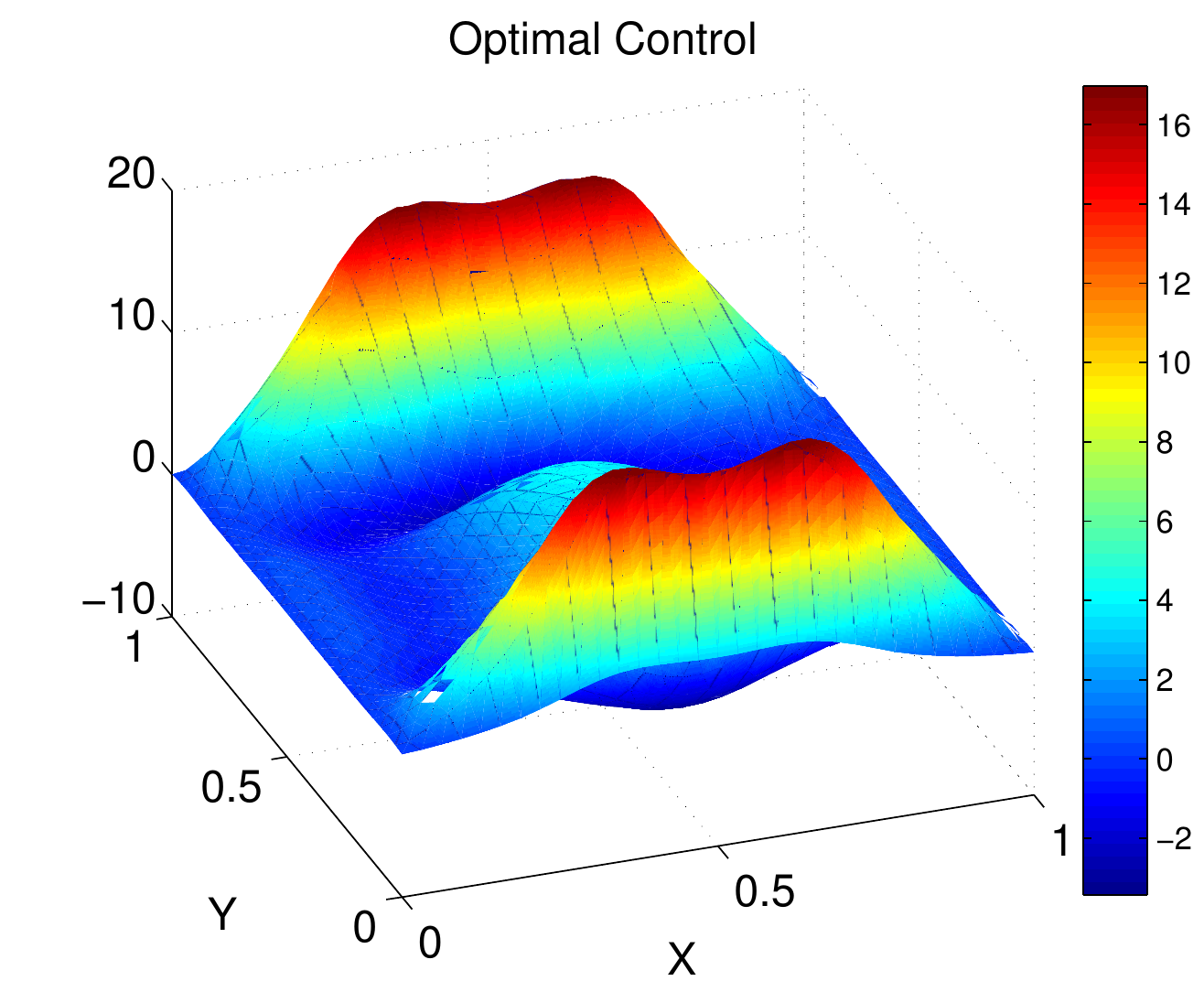}}
\subfloat{


\includegraphics[width=3.2in]{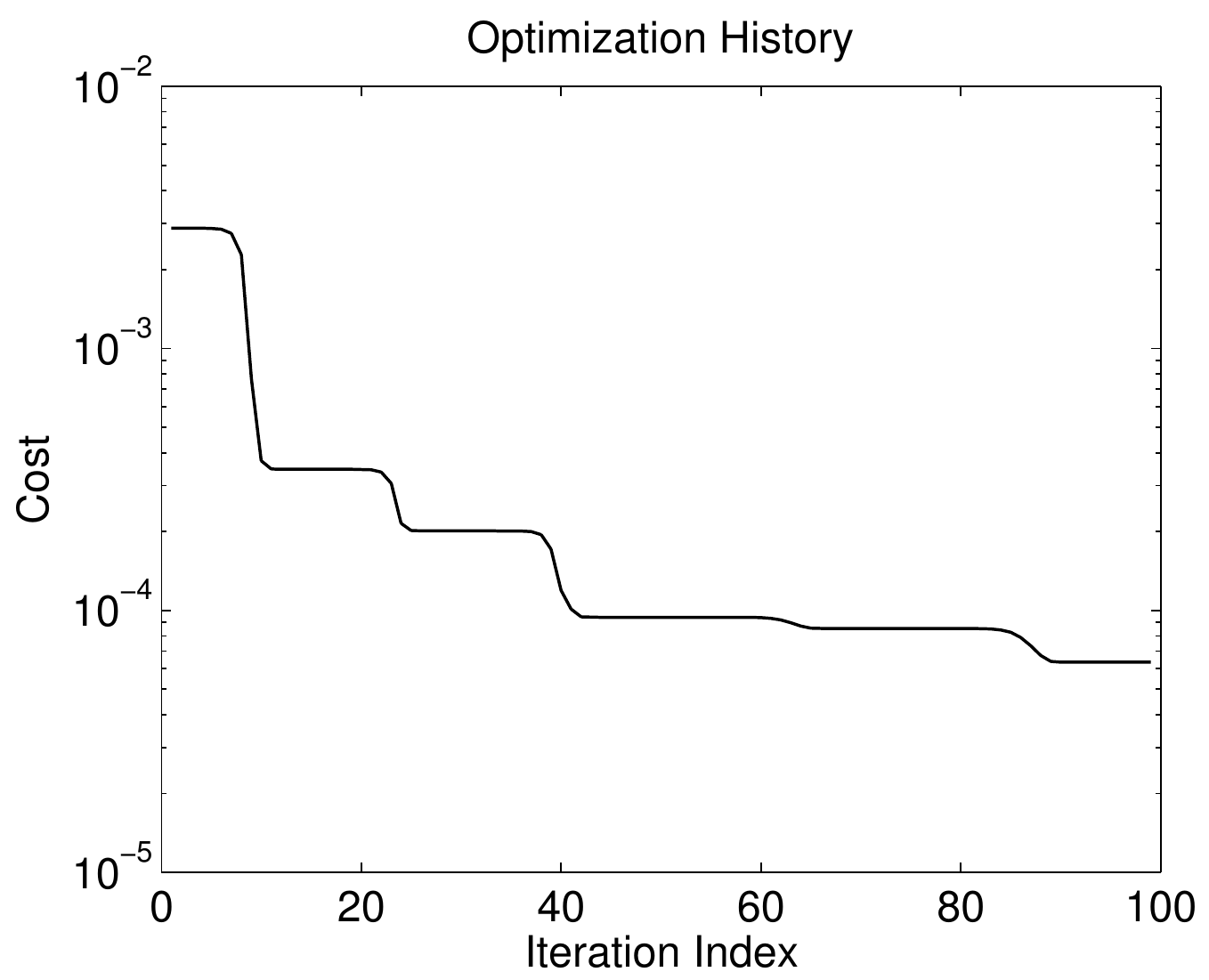}}
\caption{Left: Optimal control function $\uop$ for $y_d$ in Figure \ref{fig:Gaussian_On_Square_Nonzero_BC_Surface}.  Right: Decrease of cost functional $\cJ$.}
\label{fig:Gaussian_On_Square_Nonzero_BC_Optim}
\end{center}
\end{figure}

\subsection{Cosine On A Clover}

We take $y_d$ to be a product of cosine functions and set the boundary data to $v = 0$.  The domain $\Omega$ is a four-leaf clover (smooth domain).  See Figures \ref{fig:Cosine_On_Clover_Surface} and \ref{fig:Cosine_On_Clover_Optim} for plots of $y_d$, $\yop$, $\uop$, and the optimization history. This example also has a mismatch between the imposed boundary condition $v$ and $y_d$.  Again, the optimal surface $\yop$ matches $y_d$ well in the interior of $\Omega$, but not at the boundary.  Moreover, in Figure \ref{fig:Cosine_On_Clover_Optim}, it is evident from the convergence history of the optimization algorithm that the path to the optimal control is non-trivial.
\begin{figure}
\begin{center}
\subfloat{


\includegraphics[width=3.3in]{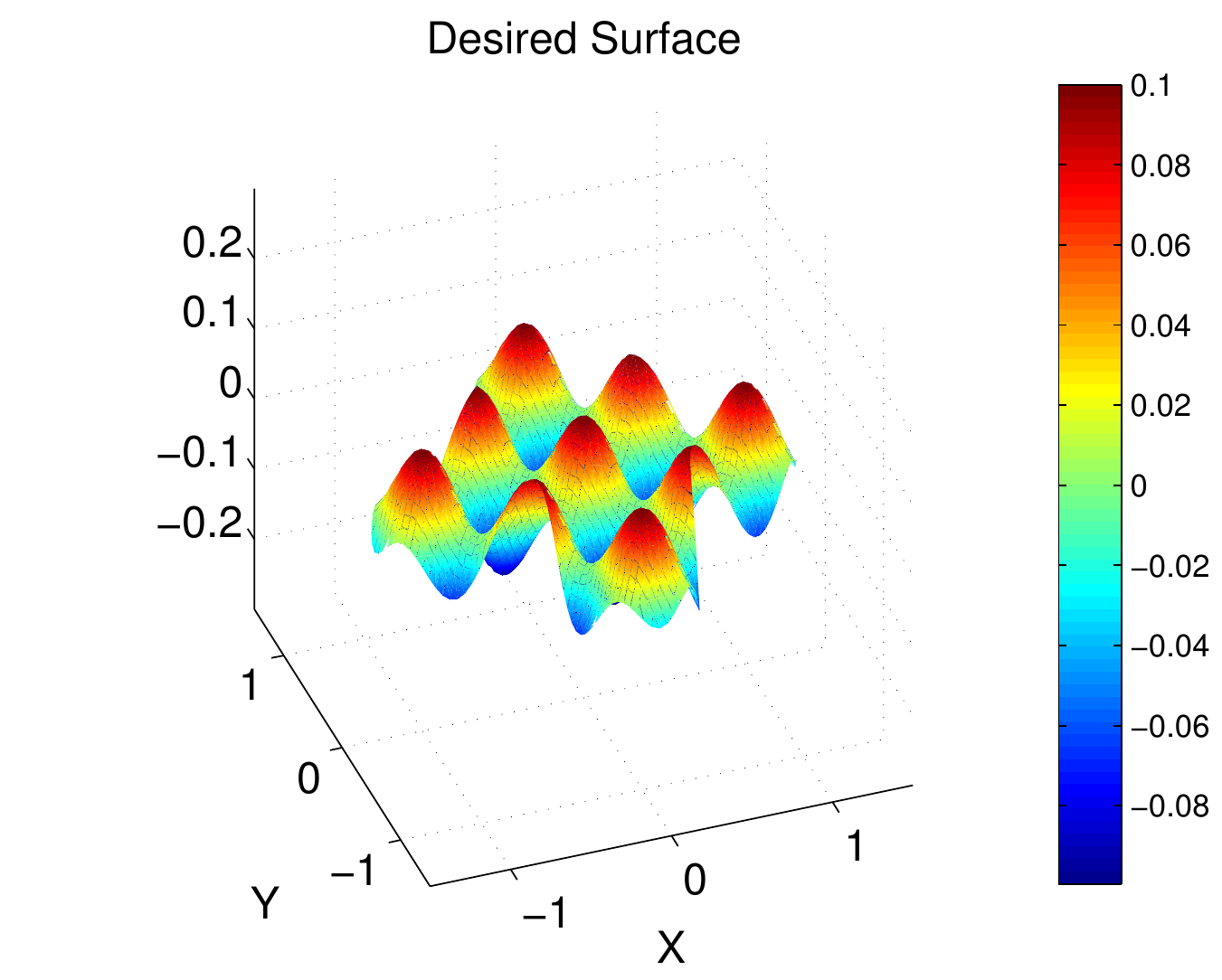}}
\hspace{-0.3in}
\subfloat{


\includegraphics[width=3.3in]{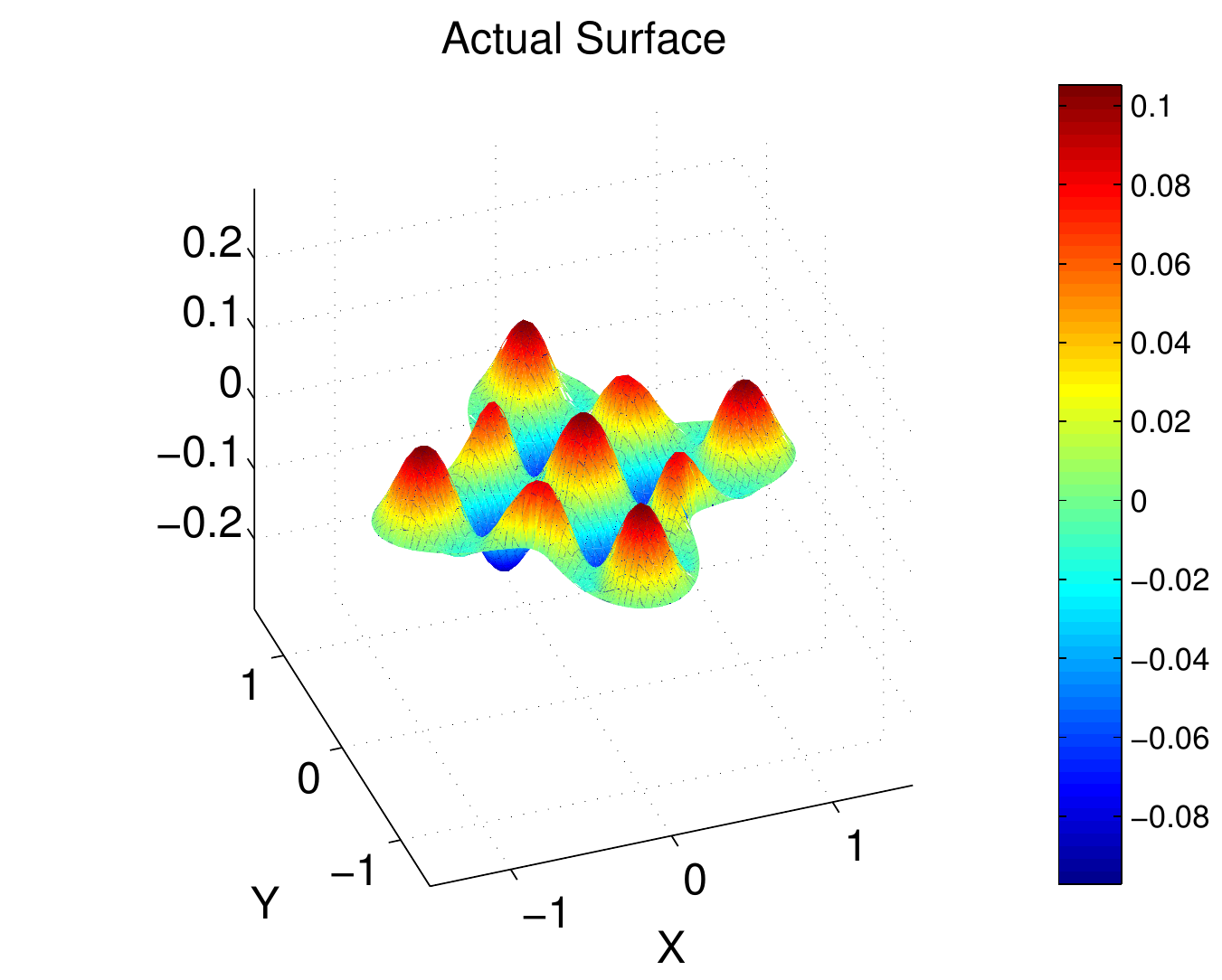}}
\caption{Left: Desired surface height $y_d = 0.1 \cos(2 \pi x) \cos(2 \pi y)$.  Right: Actual surface height $\yop$ (after the optimization method converges).  Boundary data is $v = 0$.}
\label{fig:Cosine_On_Clover_Surface}
\end{center}
\end{figure}
\begin{figure}
\begin{center}
\subfloat{


\includegraphics[width=3.2in]{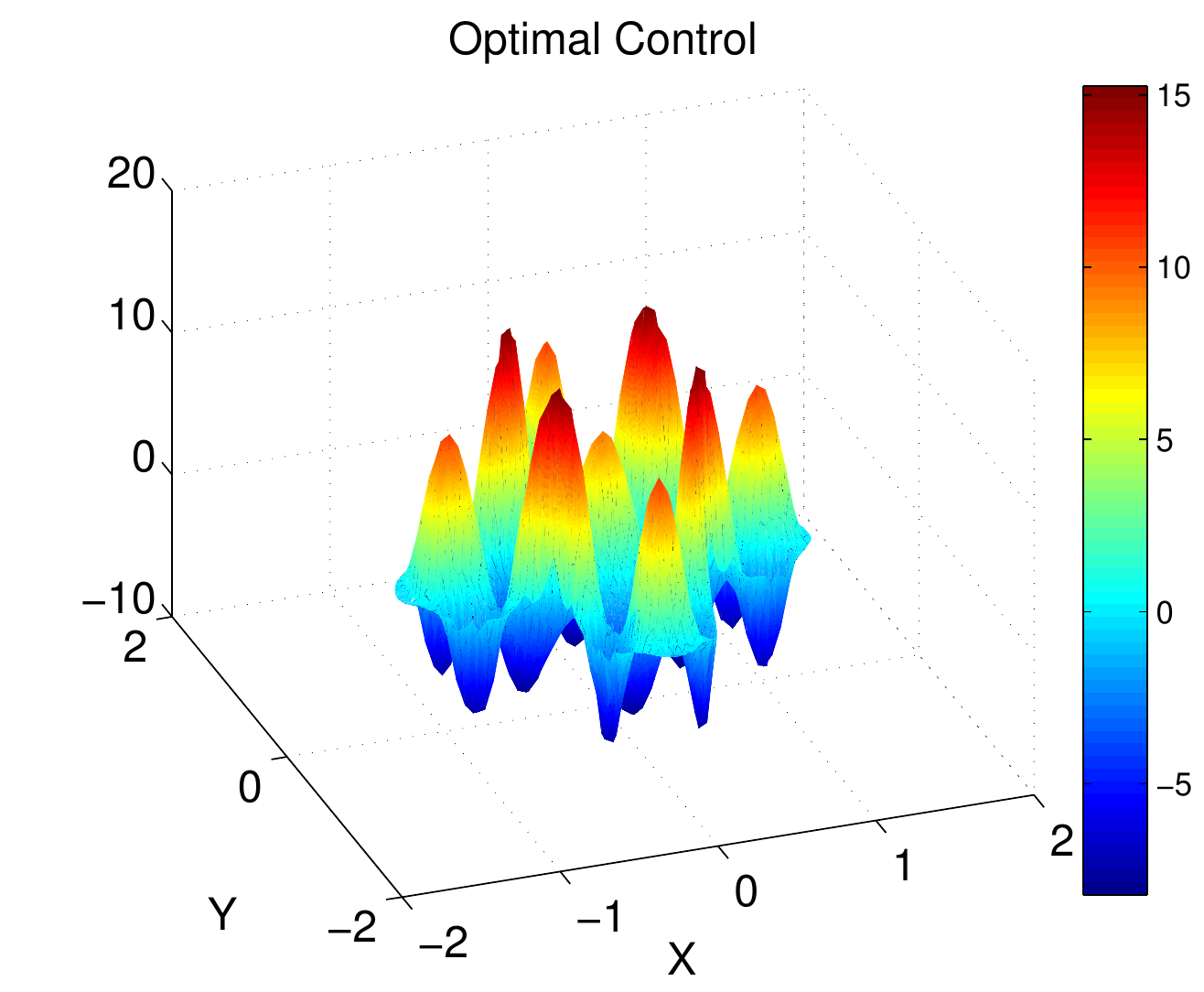}}
\subfloat{


\includegraphics[width=3.2in]{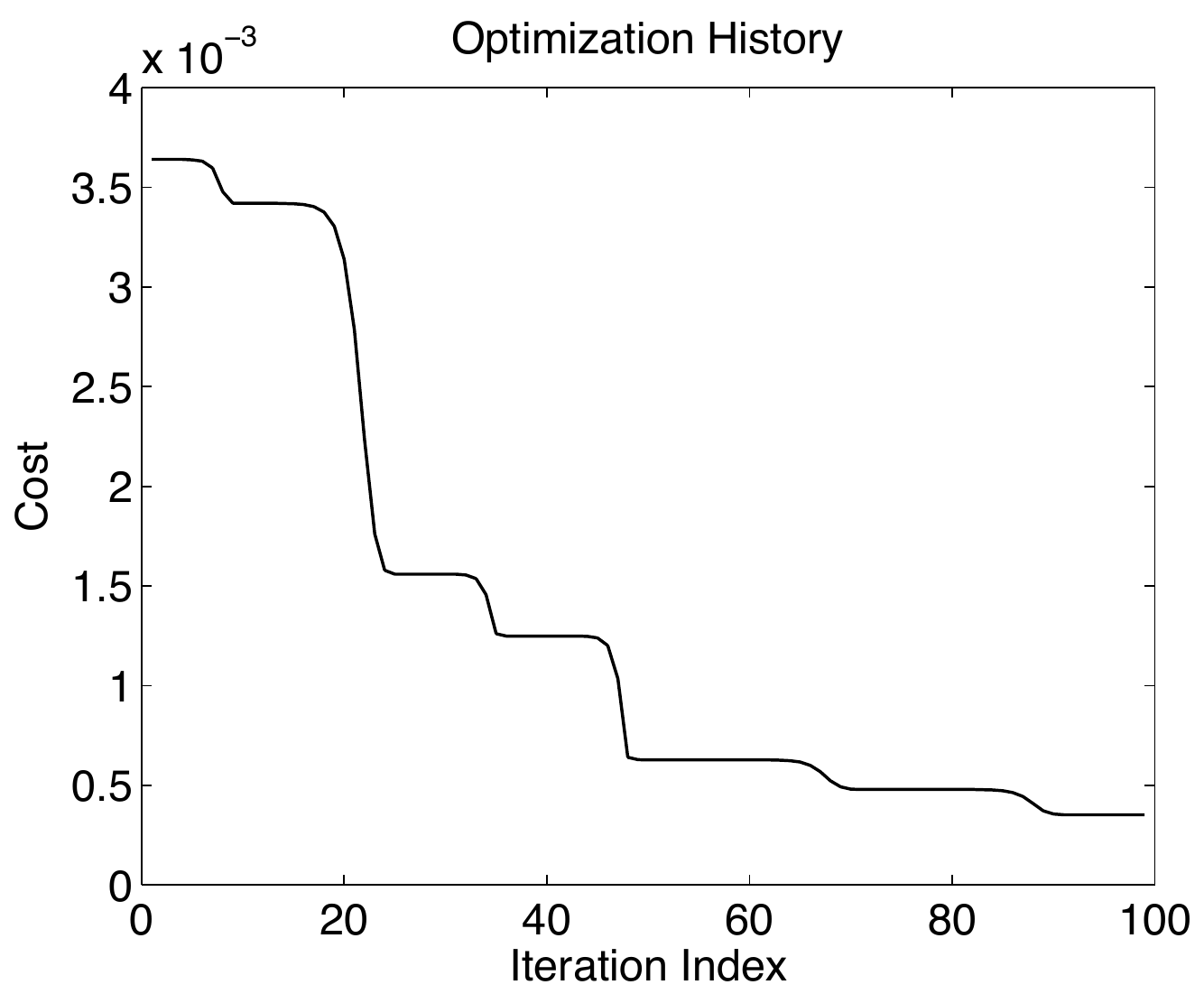}}
\caption{Left: Optimal control function $\uop$ for $y_d$ in Figure \ref{fig:Cosine_On_Clover_Surface}.  Right: Decrease of cost functional $\cJ$.}
\label{fig:Cosine_On_Clover_Optim}
\end{center}
\end{figure}

\section{Conclusion and future work}

The mean curvature operator is only locally-coercive, which leads to several difficulties in proving the existence of solution to the PDE. Using two approaches, (i) the implicit function theorem (see \thmref{thm:state_strong}) and (ii) a fixed point theorem (see \thmref{thm:fixed_point}), we provide a complete second order analysis to this PDE. The fixed point approach (ii) requires a boundary data smallness condition, but no such assumption is needed in (i).  We handle (i) by proving various Fr\'echet differentiability results, where as for (ii) we prove a new result for second order elliptic PDEs in non-divergence form, where the lower order coefficients need not be bounded (for the bounded coefficient case, see \cite[Theorem~9.15]{DGilbarg_NTrudinger_2001a}).

By using the regularity results for the PDE, we rigorously justify the first and second order sufficient optimality conditions and further tackle the \emph{2-norm discrepancy} in the $L^p$-$L^2$ pair.
We discretize the PDE using a finite element method and prove quasi-optimal error estimates for the optimal control.

There are some possible extensions of this work. The first could be boundary control. The second is where the surface tension coefficient $K \in \real^{n\times n}$ in the operator
\[
    -\divg{ K \frac{\grad y}{\cQ(y)}}
\]
acts as an optimal control, and the right-hand-side $u$ acts as a driving force.  This would be especially applicable to material science, where the presence of colloidal particles on a surface, or interface, can modulate surface tension.

\bibliographystyle{plain}
\bibliography{references}

\def\ocirc#1{\ifmmode\setbox0=\hbox{$#1$}\dimen0=\ht0 \advance\dimen0
  by1pt\rlap{\hbox to\wd0{\hss\raise\dimen0
  \hbox{\hskip.2em$\scriptscriptstyle\circ$}\hss}}#1\else {\accent"17 #1}\fi}
  \def\cprime{$'$} \def\cprime{$'$}
\begin{thebibliography}{10}

\bibitem{PAmster_MCMariani_2001a}
P.~Amster and M.~C. Mariani.
\newblock The prescribed mean curvature equation with {D}irichlet conditions.
\newblock {\em Nonlinear Anal.}, 44(1, Ser. A: Theory Methods):59--64, 2001.

\bibitem{PAmster_MCMariani_2003a}
P.~Amster and M.~C. Mariani.
\newblock The prescribed mean curvature equation for nonparametric surfaces.
\newblock {\em Nonlinear Anal.}, 52(4):1069--1077, 2003.

\bibitem{HAntil_RHNochetto_PSodre_2014a}
H.~Antil, R.~H. Nochetto, and P.~Sodr{\'e}.
\newblock Optimal {C}ontrol of a {F}ree {B}oundary {P}roblem: {A}nalysis with
  {S}econd-{O}rder {S}ufficient {C}onditions.
\newblock {\em SIAM J. Control Optim.}, 52(5):2771--2799, 2014.

\bibitem{HAntil_RHNochetto_PSodre_2013a}
H.~Antil, R.~H. Nochetto, and P.~Sodr\'e.
\newblock Optimal control of a free boundary problem with surface tension
  effects: A priori error analysis.
\newblock {\em Submitted to: SIAM Journal of Numerical Analysis.
  arXiv:1402.5709}, 2014.

\bibitem{SCBrenner_RLScott_2008a}
S.~C. Brenner and L.~R. Scott.
\newblock {\em The Mathematical Theory of Finite Element Methods}, volume~15 of
  {\em Texts in Applied Mathematics}.
\newblock Springer, New York, third edition, 2008.

\bibitem{ECasas_FTroeltzsch_2002a}
E.~Casas and F.~Tr{\"o}ltzsch.
\newblock Error estimates for the finite-element approximation of a semilinear
  elliptic control problem.
\newblock {\em Control Cybernet.}, 31(3):695--712, 2002.
\newblock Well-posedness in optimization and related topics (Warsaw, 2001).

\bibitem{ECasas_FTroeltzsch_2012a}
E.~Casas and F.~Tr{\"o}ltzsch.
\newblock Second order analysis for optimal control problems: improving results
  expected from abstract theory.
\newblock {\em SIAM J. Optim.}, 22(1):261--279, 2012.

\bibitem{PGCiarlet_2002a}
P.~G. Ciarlet.
\newblock {\em The finite element method for elliptic problems}, volume~40 of
  {\em Classics in Applied Mathematics}.
\newblock Society for Industrial and Applied Mathematics (SIAM), Philadelphia,
  PA, 2002.
\newblock Reprint of the 1978 original [North-Holland, Amsterdam; MR0520174 (58
  \#25001)].

\bibitem{IEkeland_RTemam_1999a}
I.~Ekeland and R.~T{\'e}mam.
\newblock {\em Convex analysis and variational problems}, volume~28 of {\em
  Classics in Applied Mathematics}.
\newblock Society for Industrial and Applied Mathematics (SIAM), Philadelphia,
  PA, english edition, 1999.
\newblock Translated from the French.

\bibitem{AErn_JLGuermond_2006a}
A.~Ern and J.-L. Guermond.
\newblock Evaluation of the condition number in linear systems arising in
  finite element approximations.
\newblock {\em M2AN Math. Model. Numer. Anal.}, 40(1):29--48, 2006.

\bibitem{LCEvans_1998a}
L.~C. Evans.
\newblock {\em Partial differential equations}, volume~19 of {\em Graduate
  Studies in Mathematics}.
\newblock American Mathematical Society, Providence, RI, 1998.

\bibitem{Furst_PNAS2011}
E.~M. Furst.
\newblock Directing colloidal assembly at fluid interfaces.
\newblock {\em Proceedings of the National Academy of Sciences},
  108(52):20853--20854, 2011.

\bibitem{MGiaquinta_1974}
M.~Giaquinta.
\newblock On the {D}irichlet problem for surfaces of prescribed mean curvature.
\newblock {\em Manuscripta Math.}, 12:73--86, 1974.

\bibitem{DGilbarg_NTrudinger_2001a}
D.~Gilbarg and N.~S. Trudinger.
\newblock {\em Elliptic Partial Differential Equations of Second Order}.
\newblock Classics in Mathematics. Springer-Verlag, Berlin, 2001.
\newblock Reprint of the 1998 edition.

\bibitem{MHinze_RPinnau_MUlbrich_SUlbrich_2009a}
M.~Hinze, R.~Pinnau, M.~Ulbrich, and S.~Ulbrich.
\newblock {\em Optimization with {PDE} constraints}, volume~23 of {\em
  Mathematical Modelling: Theory and Applications}.
\newblock Springer, New York, 2009.

\bibitem{Irvine_SM2012}
W.~T.~M. Irvine and V.~Vitelli.
\newblock Geometric background charge: dislocations on capillary bridges.
\newblock {\em Soft Matter}, 8:10123--10129, 2012.

\bibitem{Irvine_Nature2010}
W.~T.~M. Irvine, V.~Vitelli, and P.~M. Chaikin.
\newblock Pleats in crystals on curved surfaces.
\newblock {\em Nature}, 468(7326):947 -- 951, Dec 2010.

\bibitem{DKindelehrer_GStampacchia_1980}
D.~Kinderlehrer and G.~Stampacchia.
\newblock {\em An introduction to variational inequalities and their
  applications}, volume~88 of {\em Pure and Applied Mathematics}.
\newblock Academic Press Inc. [Harcourt Brace Jovanovich Publishers], New York,
  1980.

\bibitem{Lipowsky_PhysA1998}
R.~Lipowsky, H.-G. D{\"o}bereiner, C.~Hiergeist, and V.~Indrani.
\newblock Membrane curvature induced by polymers and colloids.
\newblock {\em Physica A: Statistical Mechanics and its Applications},
  249(1--4):536 -- 543, 1998.

\bibitem{LNirenberg_1974a}
L.~Nirenberg.
\newblock {\em Topics in nonlinear functional analysis}.
\newblock Courant Institute of Mathematical Sciences New York University, New
  York, 1974.
\newblock With a chapter by E. Zehnder, Notes by R. A. Artino, Lecture Notes,
  1973--1974.

\bibitem{RHNochetto_1989a}
R.~H. Nochetto.
\newblock Pointwise accuracy of a finite element method for nonlinear
  variational inequalities.
\newblock {\em Numer. Math.}, 54(6):601--618, 1989.

\bibitem{RScott_SZhang_1990}
L.~R. Scott and S.~Zhang.
\newblock Finite element interpolation of nonsmooth functions satisfying
  boundary conditions.
\newblock {\em Math. Comp.}, 54(190):483--493, 1990.

\bibitem{FTroltzsch_2010a}
F.~Tr{\"o}ltzsch.
\newblock {\em Optimal control of partial differential equations}, volume 112
  of {\em Graduate Studies in Mathematics}.
\newblock American Mathematical Society, Providence, RI, 2010.
\newblock Theory, methods and applications, Translated from the 2005 German
  original by J{\"u}rgen Sprekels.

\bibitem{Vitelli_SM2013}
V.~Vitelli and W.~Irvine.
\newblock The geometry and topology of soft materials.
\newblock {\em Soft Matter}, 9:8086--8087, 2013.

\bibitem{FELICITY_REF}
S.~W. Walker.
\newblock {FELICITY}: {F}inite {EL}ement {I}mplementation and {C}omputational
  {I}nterface {T}ool for {Y}ou.
\newblock http://www.mathworks.com/matlabcentral/fileexchange/31141-felicity.

\bibitem{Wang_APL2008}
J.-S. Wang, X.-Q. Feng, G.-F. Wang, and S.-W. Yu.
\newblock Twisting of nanowires induced by anisotropic surface stresses.
\newblock {\em Applied Physics Letters}, 92(19):--, 2008.

\bibitem{RMWilson_2012a}
R.M. Wilson.
\newblock Custom shapes from swell gels.
\newblock {\em Physics Today}, 65(5):15, 2012.

\bibitem{Zakhary_SM2013}
M.~J. Zakhary, P.~Sharma, A.~Ward, S.~Yardimici, and Z.~Dogic.
\newblock Geometrical edgeactants control interfacial bending rigidity of
  colloidal membranes.
\newblock {\em Soft Matter}, 9:8306--8313, 2013.

\end{thebibliography}
\end{document}